\documentclass[review,11 pt]{elsarticle}

\usepackage{color}
\usepackage{float}
\usepackage{geometry}
\usepackage{subcaption}
\usepackage{tabularx}
\usepackage{appendix}
\usepackage{enumitem}   

\usepackage{makecell}
\usepackage{amssymb}
\usepackage{amsmath}
\usepackage{amsthm}
\usepackage{multirow}
\usepackage{mathtools}
\usepackage{mathptmx}    
\usepackage{rotating}
\usepackage{mathtools}
\usepackage{mathptmx}    
\usepackage{rotating}
\usepackage{lineno}
\modulolinenumbers[5]
\usepackage{verbatim}
\usepackage{mathrsfs}
\usepackage{graphicx}
\usepackage{booktabs}
\usepackage{marginnote}
\usepackage{setspace}
\usepackage[hidelinks]{hyperref}
\usepackage{etoolbox}

\usepackage[ruled, vlined, linesnumbered]{algorithm2e}
\theoremstyle{definition}
\newtheorem{exmp}{Example}
\journal{}

\usepackage[utf8]{inputenc}
\usepackage[english]{babel}
\newtheorem{theorem}{Theorem}
\newtheorem{lemma}{Lemma}
\newtheorem{corollary}{Corollary}
\newtheorem{proposition}{Proposition}

\newtheorem{definition}{Definition}
\newtheorem{remark}{Remark}
\newtheorem{observation}{Observation}
\usepackage[numbers]{natbib}

\makeatother

\usepackage[english]{babel}

\makeatletter
\patchcmd{\ps@pprintTitle}
  {Elsevier}
  {arXiv.org}
  {}{}
\let\today\relax

\makeatother

\begin{document}

\doublespacing
\begin{frontmatter}

\title{The Outcome Range Problem in Interval Linear Programming}

\author[1]{Mohsen Mohammadi\corref{cor1}}
\ead{m0moha15@louisville.edu}
\cortext[cor1]{Corresponding author}
\author[1,2]{Monica Gentili}
\ead{monica.gentili@louisville.edu}
\address[1]{Department of Industrial Engineering, University of Louisville, J. B. Speed Building, 2301 S 3rd St, Louisville, KY 40292, United States}
\address[2]{Department of Mathematics, University of Salerno, Via Giovanni Paolo II, 132, Fisciano 84084, SA, Italy}

\begin{abstract}

Quantifying extra functions, herein referred  to as {\it outcome functions}, over optimal solutions of an optimization problem can provide decision makers with additional information on a system. This bears more importance when the optimization problem is subject to uncertainty in input parameters. In this paper, we consider linear programming problems in which input parameters are described by real-valued intervals, and we  address the {\it outcome range problem} which is the problem of finding the range of an outcome function over all possible optimal solutions of a linear program with interval data. We give a  general definition of the  problem and then focus on a special class of it where uncertainty occurs only in the right-hand side of the underlying linear program.  We show that our problem is computationally hard to solve and also study some of its theoretical properties. We then develop two approximation methods to solve it: a local search algorithm and a super-set based method. We test the methods on a set of randomly generated instances. We also provide a real case study on healthcare access measurement to show the relevance of our problem for reliable decision making.
\end{abstract}

 \date{\today}

\end{frontmatter}

\doublespacing
\setlength{\abovedisplayskip}{0.1 pt}
\setlength{\belowdisplayskip}{0.1 pt}

\section{Introduction}

In real life problems, we are sometimes interested in evaluating additional functions of interest over the results of an optimization model, that is, we are interested in evaluating functions of optimal decisions.  Let us consider, for instance, an optimization model developed to design a new transportation network. A possible function of interest, in addition to a cost function which would be optimized, could  be an environmental cost  function, useful to evaluate how the optimal transportation network impacts surrounding areas. As another example, decisions regarding the optimal location of clinics in a given region,  while  they can improve public health in a community, might in turn lead to undesirable consequences on a larger scale, such as disparities in access to healthcare among different communities.
We refer to the additional functions of interest as {\it outcome functions}, which are used to evaluate unintended consequences of optimal decision making.

Outcome functions do not have a direct role in the decision process. They are not, in other words, the main objective function of the optimization model, whereas  they  might have a significant  role in providing important information for future decisions or actions. 
This is particularly relevant for government agencies, public health decision makers, policy makers, city
managers and other stakeholders who make decisions that have differential impacts on different
communities and sub-populations. For example, Nobles et al. \cite{nobles2014spatial} and Gentili et al. \cite{gentili2016projecting,gentili2018quantifying,gentili2015small}  used outcome functions to evaluate  spatial access to  pediatric and adult primary care. They  developed an optimization model for matching patients and providers, and defined two linear outcome functions to quantify spatial access to healthcare services. In another study, Zheng et al. \cite{zheng2017regularized} presented an application  in telecommunication networks, where one is interested in designing the network such that enough band-width is allocated between two nodes in order to minimize the total demand lost. An outcome function of interest, in this context, is the local performance of each node defined as the volume of unmet requests from the node.

Quantifying the impact of decisions using outcome functions becomes even more relevant when decisions
are made in an uncertain environment, which is the focus of this paper. 
Uncertainty in  optimization problems usually derives from
uncertainty in input parameters, occurring  due to measurement
errors, missing data, rounding errors, statistical estimations, etc. Solutions to optimization problems can exhibit considerable sensitivity to perturbations in the input parameters, thus often returning a solution which is highly infeasible and/or suboptimal \cite{bertsimas2004price}. 

Throughout the years, several approaches to treat uncertainty in input data have emerged such as robust optimization, stochastic optimization, parametric programming, fuzzy programming, and interval optimization, depending on the source of uncertainty and the requirements on the returned solution. In this paper, we adopt the approach of interval linear programming (ILP) where we assume that input parameters can vary within a-priori known  intervals. Several topics have been subject of research in this area (see \cite{survey} for a comprehensive survey on the topics): (i) Oettli and Prager \cite{oettli1964compatibility} and Rohn \cite{rohn2} addressed the problem of characterizing the set of all  possible feasible solutions; (ii) Novotn{\'a} et al. \cite{novotna2020duality} studied the duality gap problem in interval linear programming; (iii) the  problem of describing the set of all possible optimal solutions was  studied by Allahdadi and Nehi \cite{allahdadi2013optimal} and later by Garajov{\'a} and Hlad{\'\i}k \cite{garajova2019optimal} and  also its approximation was discussed by \cite{contractor, 10.1145/3396474.3396479,jansson1991rigorous}; (iv)  the problem of determining a satisficing solution space was subject of study in \cite{wang2014violation,zhou2009enhanced}. A problem of particular interest, because of its relevance from an application perspective, is that of  finding the range of  optimal values  of an interval linear program, known in the  literature as the \textit{optimal value range problem}.  The exact formulation and characterization of the problem was discussed in \cite{chinneck2000linear,hladik2009optimal,mraz1998calculating,rohn1}, while \cite{hladik2014approximation,soco} developed some approximation algorithms for the intractable cases. A different approach to  get a satisficing optimal value range was  investigated by \cite{huang1995grey}. The optimal value range problem has been applied in several application problems, such as transportation problems with interval supply and demand \cite{itp,d2020optimal,juman2014heuristic}, matrix games with interval-valued payoffs  \cite{li2016interval,liu2009matrix}, and portfolio selection problems with interval approximations  of expected returns \cite{kumar2016interval,lai2002class}.

In this context, our focus is on studying a problem close to the optimal value range problem where we are interested in determining the range of an outcome function (other than the objective function) associated with an interval linear program. To this aim, we introduce the {\it Outcome Range Problem} which consists of determining the minimum and
the maximum values of a given (additional) linear function over the set of all possible optimal solutions of an interval-valued linear program. We formally define our problem, analyze its relation to and differences with the optimal value range problem, and  study a specific case  where uncertainty occurs only in the right-hand side of the underlying linear program. We show that solving the outcome range problem  to optimality is not an easy task; we then study some theoretical properties of the problem and develop two solution approaches to approximate the optimal values. We evaluate our solution techniques on a set of randomly generated instances, and finally, to outline the relevance of our problem for reliable decision making, we present a case study where we apply our approach to quantify spatial access to healthcare services.

The remainder of the paper is structured as follows. We first present an introductory example to motivate our problem. We then introduce some basic notations, and formally define the outcome range problem in Section \ref{oip}. We assess the  computational complexity of the problem in Section \ref{properties}. In Section \ref{sec:prop}, we explore theoretical properties of our problem. We describe our  solution techniques in Section \ref{techniques}. In Section \ref{results}, we discuss results of our experimental study.  Section \ref{case} presents an healthcare application of our problem. Finally, we summarize our findings in Section \ref{conclusion}.

\section{Quantifying an Outcome Function  Under Uncertainty: An Introductory Example} \label{out}
Let us consider the classical transportation problem \cite{winston2003introduction} where the main goal is to decide how to transfer goods from a set of $m$ origins to a set of $n$ destinations with minimal cost such that the capacity at each origin is not exceeded, and the demand at each destination is satisfied.
The transportation problem can be formulated as follows

\begin{eqnarray}
 \min \; \sum_{i=1}^{n}\sum_{j=1}^{m} c_{ij}x_{ij} \label{objTransportation}\\
\text{subject to} \nonumber\\
 \sum_{j=1}^{m} x_{ij} &\leq s_i, & \forall i=1,...,n, \label{con:capacity}\\
 \sum_{i=1}^{n} x_{ij} &\geq d_j, & \forall j=1,...,m, \label{con:demand}\\
 x_{ij} &\geq 0, &  \forall i=1,...,n,\;j=1,...,m, \label{con:nonneg}
\end{eqnarray}
where $x_{ij}$ is a decision variable which determines the size of the shipment from origin $i$ to destination $j$, $c_{ij}$ is the unit shipping cost from origin  $i$ to destination $j$, $s_i$ is the total supply of origin $i$, and $d_j$ is the total demand of destination $j$. The objective function of the model minimizes the total transportation cost. The two sets of constraints ensure that the resulting transportation plan respects the capacity at each origin (Eq. (\ref{con:capacity})), and meets the demand of each destination (Eq. (\ref{con:demand})). Let us consider a specific instance of the problem where there are three origins and three destinations (see Table \ref{parameters} for shipping costs,  supply and  demand levels).  
\begin{table}[H] 
\caption{Shipping costs, supply and demand levels.}
\centering
\begin{tabular}{@{}ccccc@{}}
\toprule
            & \multicolumn{3}{c}{to}         & \multirow{2}{*}{supply (ton)} \\ \cmidrule(lr){2-4}
from        & destination 1 & destination 2 & destination 3 &                         \\ \midrule
origin 1 & \$40       & \$21       & \$23       & 70                     \\
origin 2 & \$24       & \$43       & \$19       & 75                     \\
origin 3 & \$31       & \$35      & \$21       & 81                     \\ \midrule
demand (ton)      & 85      & 64      & 71       &                         \\ \bottomrule
\end{tabular}
\label{parameters}
\end{table}
The optimal shipping cost, considering the input data in Table \ref{parameters}, is \$4,945 and the optimal solution to the problem is shown in Figure \ref{fig:transportation}, where labels on each arc denote the total quantity shipped on the arc. We can associate with the transportation problem an outcome function to evaluate, for example, the environmental  impact \cite{park2016environmental,zhou2014measuring} of the  optimal transportation plan, as total pounds of CO$_2$ emissions. 
\begin{figure} [h] 
\centering
    \includegraphics[scale=0.65]{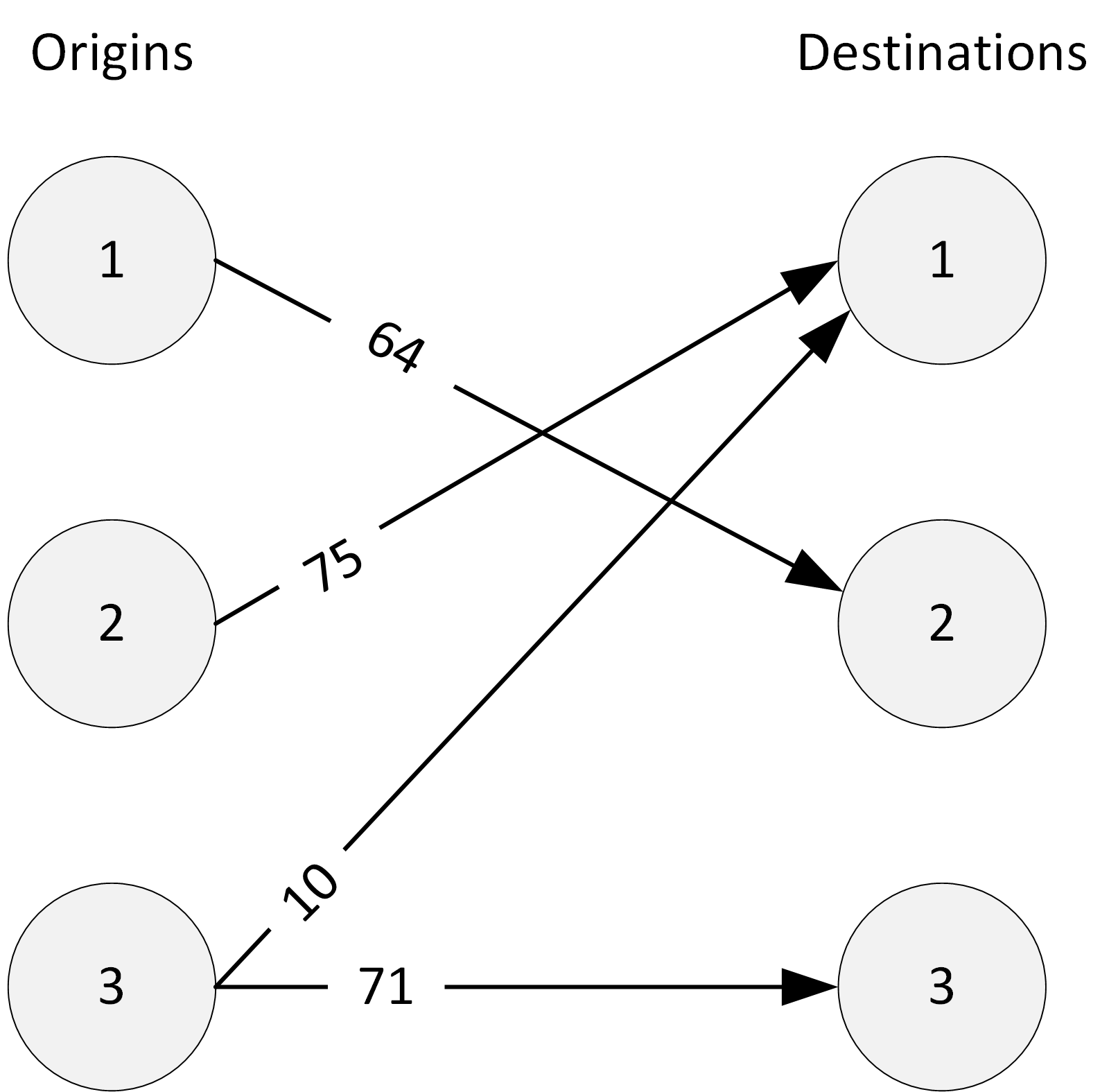}
\caption{The optimal transportation plan for model (\ref{objTransportation}-\ref{con:nonneg}) with input data as described in Table \ref{parameters}.}
\label{fig:transportation}
\end{figure}
The CO$_2$ emissions depend on the amount of fuel consumed to transport the products to destinations, and consequently varies with the travel distance and with the amount of products. Let $r_{ij}$ denote the total pounds of CO$_2$ emitted in the atmosphere per unit of the product shipped from origin $i$ to destination $j$ (the specific values of these parameters for our example are reported in Table \ref{table:co2}), and let $f(x) = \sum_{i,j}r_{ij}x_{ij}$ be an outcome function associated with a given transportation problem.
The value of this outcome function on the optimal transportation plan for our example is equal to 3,940 $lb$.

\begin{table}[h]
\centering
\caption{The CO$_2$ emission associated with the arcs of the transportation network.}

\begin{tabular}{@{}cccc@{}}
\toprule
            & \multicolumn{3}{c}{to}         \\ \cmidrule(l){2-4} 
from        & destination 1 & destination 2 & destination 3 \\ \midrule
origin 1 & 30 $lb$      & 17 $lb$      & 18 $lb$      \\
origin 2 & 19 $lb$     & 32 $lb$      & 14 $lb$     \\
origin 3 & 22 $lb$      &25 $lb$      & 17 $lb$      \\ \bottomrule
\end{tabular}
\label{table:co2}
\end{table}
Now let us assume that the demands are not known with certainty, but rather they vary in given intervals. Then the mathematical formulation reads 
\begin{eqnarray}
\min \; \sum_{i=1}^{n}\sum_{j=1}^{m} c_{ij}x_{ij} \nonumber\\
\text{subject to} \nonumber\\
\sum_{j=1}^{m} x_{ij} &\leq s_i , & \forall i=1,...,n, \nonumber\\
\sum_{i=1}^{n} x_{ij} &\geq  [\underline{d}_j, \overline{d}_j] , & \forall j=1,...,m ,\nonumber\\
x_{ij} &\geq 0, & \forall i=1,...,n,\;j=1,...,m, \nonumber
\end{eqnarray}

\noindent where $[\underline{d}_j, \overline{d}_j]$ is the range of values which can be assumed by the demand at destination $j$, for all $j$. The  question we would like to address is:  \textit{how does  uncertainty in the parameters affect the environmental cost?} That is, \textit{how does the environmental cost (the total CO$_2$ emission) change when the parameters change?} If we apply one of the most commonly used approaches to address uncertainty in optimization models such as, for example, robust optimization, we would only be able to evaluate the outcome function on a single robust solution \cite{soyster1973convex} or a number of solutions with some level of protection against uncertainty in the data \cite{bertsimas2004price}. However, such an evaluation would not answer our question of quantifying the variation of the  outcome function in response to uncertainty in the parameters. A much more useful information would be, for example, the range of variation of the  outcome function, that is, the {\it best} and \textit{worst} values of the outcome function over the set of {\it all} the optimal solutions corresponding to all realizations of the uncertain data. 
\begin{table}[H]
\caption{All the possible realizations of the uncertain demand with the corresponding optimal solutions and values of $f(x)$ for our  transportation problem example.}
\tabcolsep=0.2cm
\renewcommand\arraystretch{0.72}
\small
\centering
\begin{tabular}{@{}cccccccccccc@{}}
\toprule
 &  & \multicolumn{9}{c}{optimal solutions} &  \\ \cmidrule(lr){3-11}
\makecell{data \\realization }& demand values & $x_{11}$ & $x_{12}$ & $x_{13}$ & $x_{21}$ & $x_{22}$ & $x_{23}$ & $x_{31}$ & $x_{32}$ & $x_{33}$ & $f(x)$ \\ \midrule
1 & $\{d_1=85,d_2=64,d_3=71\}$ & 0 & 64 & 0 & 75 & 0 & 0 & 10 & 0 & 71 &3,940 $lb$\\
2 & $\{d_1=85,d_2=64,d_3=72\}$ & 0 & 64 & 1 & 75 & 0 & 0 & 10 & 0 & 71 & 3,958 $lb$\\
3 & $\{d_1=85,d_2=64,d_3=73\}$ & 0 & 64& 2 & 75 & 0 & 0 & 10 & 0 & 71& 3,976 $lb$\\
4 & $\{d_1=85,d_2=65,d_3=71\}$ & 0 & 65 & 0 & 75 & 0 & 0 & 10 & 0 & 71&3,957 $lb$\\
5 & $\{d_1=85,d_2=65,d_3=72\}$ & 0 & 65 & 1 & 75 & 0 & 0 & 10 & 0 &71 &3,975$lb$\\
6 & $\{d_1=85,d_2=65,d_3=73\}$ & 0 & 65 & 2 & 75 & 0 & 0 & 10 & 0 &71 &3,993 $lb$\\
7 & $\{d_1=85,d_2=66,d_3=71\}$ & 0 & 66 & 0 & 75 & 0 & 0 & 10 & 0 & 71&3,974 $lb$\\
8 & $\{d_1=85,d_2=66,d_3=72\}$ & 0 & 66 & 1 & 75 & 0 & 0 & 10 & 0 & 71&3,992 $lb$\\
9 & $\{d_1=85,d_2=66,d_3=73\}$ & 0 & 66 & 2 & 75 & 0 & 0 & 10 & 0 & 71&4,010 $lb$\\
10 & $\{d_1=86,d_2=64,d_3=71\}$ & 0 & 64 & 1 & 75 & 0 & 0 & 11 & 0 & 70 &3,963 $lb$\\
11 & $\{d_1=86,d_2=64,d_3=72\}$ & 0 & 64 & 2 & 75 & 0 & 0 & 11 & 0 & 70 &3,981 $lb$\\
12 & $\{d_1=86,d_2=64,d_3=73\}$ & 0 & 64 & 3 & 75 & 0 & 0 & 11 & 0 & 70 &3,999 $lb$\\
13 & $\{d_1=86,d_2=65,d_3=71\}$ & 0 & 65 & 1 & 75 & 0 & 0 & 11 & 0 & 70 &3,980 $lb$\\
14 & $\{d_1=86,d_2=65,d_3=72\}$ & 0 & 65 & 2 & 75 & 0 & 0 & 11 & 0 & 70 &3,998 $lb$\\
15 & $\{d_1=86,d_2=65,d_3=73\}$ & 0 & 65 & 3 & 75 & 0 & 0 & 11 & 0 & 70 &4,016 $lb$\\
16 & $\{d_1=86,d_2=66,d_3=71\}$ & 0 & 66 & 1 & 75 & 0 & 0 & 11 & 0 & 70 &3,997 $lb$\\
17 & $\{d_1=86,d_2=66,d_3=72\}$ & 0 & 66 & 2 & 75 & 0 & 0 & 11 & 0 & 70 &4,015 $lb$\\
18 & $\{d_1=86,d_2=66,d_3=73\}$ & 0 & 66 & 3 & 75 & 0 & 0 & 11 & 0 & 70 &4,033 $lb$\\
19 & $\{d_1=87,d_2=64,d_3=71\}$ & 0 & 64 & 2 & 75 & 0 & 0 & 12 & 0 & 69 & 3,986 $lb$\\
20 & $\{d_1=87,d_2=64,d_3=72\}$ & 0 & 64 & 3 & 75 & 0 & 0 & 12 & 0 & 69 &4,004 $lb$\\
21 & $\{d_1=87,d_2=64,d_3=73\}$ & 0 & 64 & 4 & 75 & 0 & 0 & 12 & 0 & 69 &4,022 $lb$\\
22 & $\{d_1=87,d_2=65,d_3=71\}$ & 0 & 65 & 2 & 75 & 0 & 0 & 12 & 0 & 69 &4,003 $lb$\\
23 & $\{d_1=87,d_2=65,d_3=72\}$ & 0 & 65 & 3 & 75 & 0 & 0 & 12 & 0 & 69 &4,021 $lb$\\
24 & $\{d_1=87,d_2=65,d_3=73\}$ & 0 & 65 & 4 & 75 & 0 & 0 & 12 & 0 & 69 &4,039 $lb$\\
25 & $\{d_1=87,d_2=66,d_3=71\}$ & 0 & 66 & 2 & 75 & 0 & 0 & 12 & 0 & 69 &4,020 $lb$\\
26 & $\{d_1=87,d_2=66,d_3=72\}$ & 0 & 66 & 3 & 75 & 0 & 0 & 12 & 0 & 69 &4,038 $lb$\\
27 & $\{d_1=87,d_2=66,d_3=73\}$ & 0 & 66 & 4 & 75 & 0 & 0 & 12 & 0 & 69 &4,056 $lb$\\ \bottomrule
\end{tabular}
\label{table:uncertainty}
\end{table}

Going back to our example, let us assume the demand level intervals are $d_1 \in [85,87]$, $d_2 \in [64,66]$, and $d_3 \in [71,73]$. For the sake of clarity in the exposition, let us also assume that the demand at the destinations can only take integer values in the given intervals. By applying a conservative robust approach, we would look for a shipment plan which is feasible under all the possible data perturbations, and would then evaluate the outcome function on the returned robust solution. In this case for example, by applying the worst case robust approach \cite{gabrel2010linear}, we would choose to ship 87 units to destination 1, 66 units to destination 2, and 73 units to destination 3 for a total cost of \$ 5,099, and with an environmental impact equal to  4,056 $lb$.

Let us list, for this simple example, all the realizations of the uncertain data. They are shown in the first two columns of Table \ref{table:uncertainty}.  For each realization of the data, we solved the corresponding transportation problem, and evaluated the outcome function on the corresponding optimal solution. Columns 3-11 in the table report the optimal solutions, and the last column in the table reports the corresponding value of the outcome function. In this simple example, the best value of the outcome function is equal to 3,940 $lb$ (corresponding to scenario 1)  and the worst value is equal to 4,056 $lb$ (corresponding to scenario 27). Hence, in this case, we can say that given all the possible realizations of the interval data, the total CO$_2$ emission of the transportation plan would range between 3,940 $lb$ and 4,056 $lb$. As can be seen from the results, the optimal solutions are very sensitive to the demand perturbations. This makes the problem of finding the best and the worst values of $f(x)$ a nontrivial one. 

In this simple example, given a linear program with interval parameters and an associated linear outcome function, we determined the best and the worst values of the latter among all the possible optimal solutions obtained from all the realizations of the interval data. We refer to this problem as the outcome range problem. Its formal definition is given in the next section.
\section{The Outcome Range Problem} \label{oip}
Let us introduce some needed notation which is commonly used in the interval linear programming literature \cite{survey,rohn2005handbook}.
 Given two matrices $\underline{A},\overline{A} \in \mathbb{R}^{m \times n}$, we define an interval matrix as the set
$$\mathbf{A}=[\underline{A}, \overline{A}]:=\{A \in \mathbb{R}^{m \times n}:\; \underline{A} \leq A \leq \overline{A}\},$$
where matrices $\underline{A},\overline{A}$ are called the lower and the upper bounds  of $\mathbf{A}$, respectively, and comparing matrices is understood componentwise. The set of all $m$-by-$n$ real interval matrices is denoted by $\mathbb{IR}^{m\times n}$. We define an interval vector analogously. For the sake of simplicity, we write $\mathbb{IR}^{m}$ instead of $\mathbb{IR}^{m\times 1}$ to denote the set of all real interval vectors of order $m$. Throughout this paper, we use bold symbols for interval vectors and matrices.
Let us consider the following interval linear programming (ILP) problem in the form of
\begin{equation}
\label{ILP}
 \min \; \mathbf{c}^Tx \;\;\text{subject to}\;\;x \in \mathcal{M}(\mathbf{A},\mathbf{b}),
\end{equation}
where we are given $\mathbf{c} \in \mathbb{IR}^n$, $\mathbf{b} \in \mathbb{IR}^m$, and $\mathbf{A} \in \mathbb{IR}^{m \times n}$.  $\mathcal{M}(\mathbf{A},\mathbf{b})$ denotes the feasible set described by  linear constraints with the interval coefficient  matrix $\mathbf{A}$ and the interval right-hand side vector $\mathbf{b}$. 
Interval linear programming has been extensively studied with three main types of $\mathcal{M}(\mathbf{A},\mathbf{b})$, which are shown in Table \ref{types}. The type of constraints and restriction on variables in an interval linear program can considerably impact its properties. Thus, each type of interval linear programs is usually treated separately in the literature.\footnote{ References \cite{chinneck2000linear,hladi2013weak} address the general form.}
\begin{table}[H]
\centering
\caption{Different types of interval linear constraints \cite{survey}}
\label{types}
\begin{tabular}{@{}cl@{}}
\toprule
type & \multicolumn{1}{c}{interval linear system} \\ \midrule
(I) & $\mathcal{M}(\mathbf{A},\mathbf{b})= \{x\in \mathbb{R}^{n};\;\;\mathbf{A}x=\mathbf{b},\;\;x\ge 0\}$ \\
(II) & $\mathcal{M}(\mathbf{A},\mathbf{b})= \{x\in \mathbb{R}^{n};\;\;\mathbf{A}x \le \mathbf{b}\}$ \\
(III) & $\mathcal{M}(\mathbf{A},\mathbf{b})= \{x\in \mathbb{R}^{n};\;\;\mathbf{A}x\le \mathbf{b},\;\;x\ge 0\}$ \\ \bottomrule
\end{tabular}
\end{table}

We refer to any triple $(A, b, c)$,  where $A\in\mathbf{A}$,  $b\in\mathbf{b}$, and $c\in\mathbf{c}$, as a  {\it scenario}. With each scenario  $(A,b,c)$, we can associate a linear program, namely LP$(A,b,c)$, whose feasible set and optimal value are denoted by $\mathcal{M}(A,b)$ and $z(A,b,c)$, respectively, i.e.,
\begin{equation*}
 z(A,b,c) := \{\min\; c^Tx\;\;\text{subject to}\;\;x \in \mathcal{M}(A,b)\}.
\end{equation*}
Hence, an interval linear program is a family of linear programs associated with all $A \in \mathbf{A}$, $b\in \mathbf{b}$ and $c\in \mathbf{c}$. 
For a particular scenario $(A,b,c)$, the corresponding LP$(A,b,c)$ can be infeasible, unbounded or admit a finite optimal value. We denote by $s(A,b,c)$ an optimal solution (or the set of all optimal solutions) of a linear program LP$(A,b,c)$, if any, admitting a finite optimal value. We denote by $\Omega$ the set of all the optimal solutions of an interval linear program, referred to as the {\it  optimal set}, that is,
$$\Omega := \bigcup_{A \in \mathbf{A}, b\in \mathbf{b}, c\in \mathbf{c}}s(A,b,c).$$

We are now ready to formally define our problem. Given the ILP (\ref{ILP}) and an additional linear function $f: \mathbb{R} ^n \rightarrow \mathbb{R}$, where $f(x)=r^Tx$ with $r \in \mathbb{R}^{n}$, the {\it outcome range problem} consists in solving the two following optimization problems
	\begin{align}
	\underline{f}& := \{ \min\; f(x)\;\;\text{subject to}\;\; x \in \Omega \}, \nonumber \\
	\overline{f} &:= \{ \max\; f(x)\;\;\text {subject to} \;\; x \in \Omega \}. \nonumber
	\end{align}
We define the pair of optimal values $ \{ \underline{f},\overline{f}\}$ to be the optimal solution of the outcome range problem.
\begin {exmp}\label{exmp}
Consider the following two-dimensional  linear program with interval right-hand sides

$$\min \;\;(2,-5) ^{T}x\;\; \text{subject to}\; \left(\begin{array}{cc} 1& -1\\ -1 & -1\\0&1 \end{array}\right)x\le \left(\begin{array}{c} \left[4, 7\right]\\ \left[-6, 8\right]\\ \left[4,9\right] \end{array}\right),\;\; x\ge0, $$
and consider the following outcome function
$$
f(x) = 8x_1 + 9x_2.
$$
Let us consider Figure \ref{fig} where the optimal solution $\{\underline{f}, \overline{f}\}$ of the problem is shown. In the figure, the intersection and the union of all the feasible sets of  the linear programs associated with all the scenarios are shown in dark and light gray, respectively. Specifically, the intersection of all the feasible sets is obtained by setting the right-hand sides at their lower bound, while the union of all the feasible sets is obtained by setting all the right-hand sides at their upper bound. The black area represents the set $\Omega$, that is, the set of all optimal solutions obtained from all the realizations of the interval data. Both the minimum and the maximum values of $f(x)$ occur at the endpoints of the bold line and are shown in the figure. Their values are $\underline{f} = 36$ and $\overline{f} = 81$, respectively. In particular, $\overline{f}$ is obtained on the point $x^{1*} =(0,9)$ which is the optimal point of several linear programs one of which is associated, for example, with scenario $b^T =(4,-6,9) $, while $\underline{f}$ is obtained on the point $x^{2*} =(0,4)$ which is the optimal point of a linear program associated, for example, with scenario $b^T =(7,8,4)$. 
\end{exmp}

\begin{figure} [h!] 
\centering
    \includegraphics[scale=0.80]{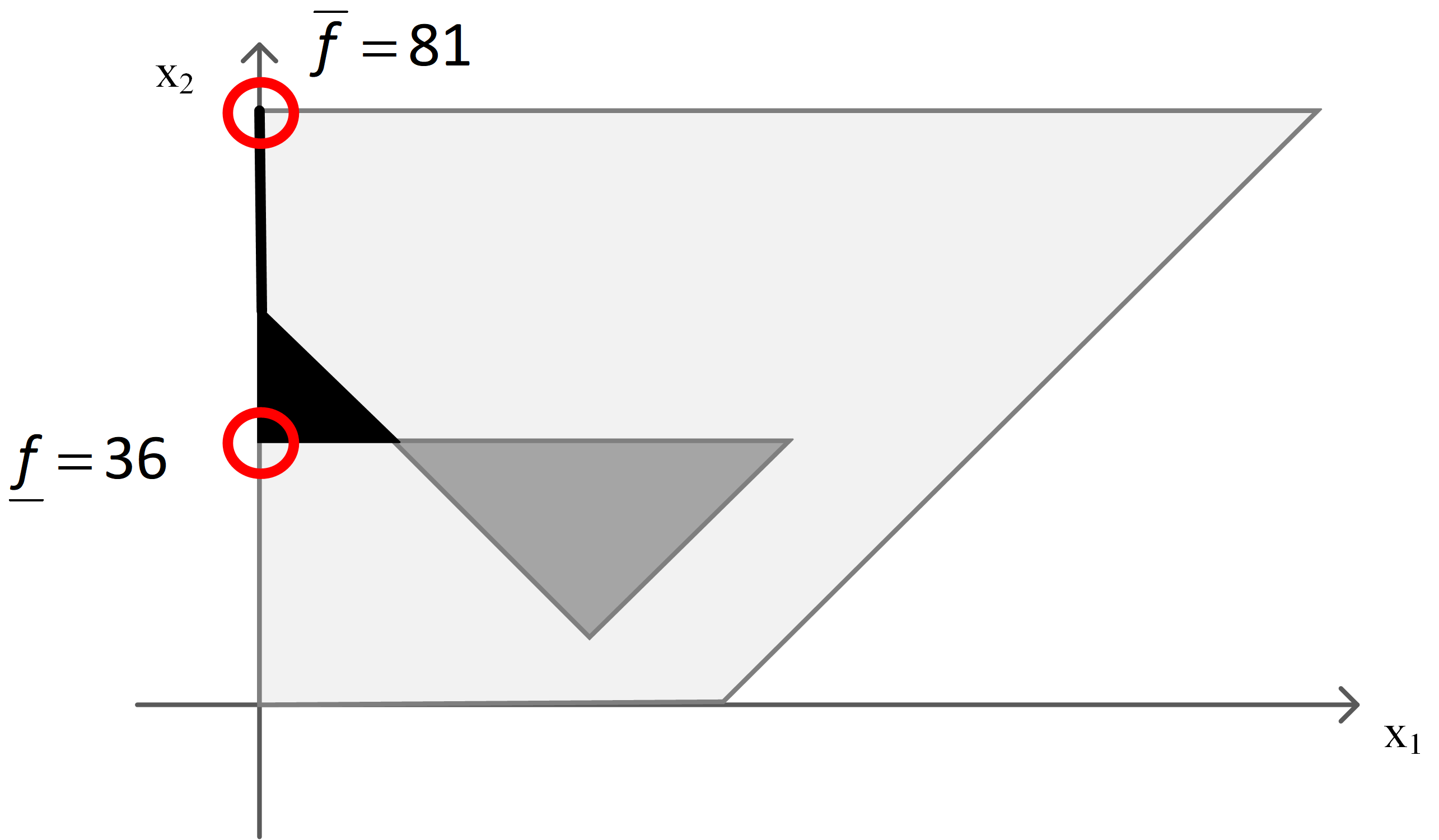}
\caption{(Example \ref{exmp}) Intersection of all feasible sets in dark gray; union in light gray; set of all optimal solutions in black.}
\label{fig}
\end{figure}

\subsection{Our Focus}
\label{sec:Focus}
As can be observed from Example \ref{exmp}, the difficulty in solving the outcome range problem relies on the fact that its feasible set, that is, the set $\Omega$, is not explicitly known; nor a convenient implicit description of it (e.g., polyhedral description)
is available in general \cite{garajova2019optimal,contractor} (with some exceptions as outlined in  Section \ref{properties}). This is true even if we consider a simplified
version of the underlying ILP where we only deal with interval right-hand sides. As we mentioned earlier, the three types of interval linear programs are analyzed separately in the literature because the feasible region and the optimal set might change when applying standard linear transformations. In the discussion to follow, we will focus on solving the outcome range problem
when the underlying interval linear program is of Type III, that is, it contains inequality and non-negativity constraints, and uncertainty occurs only in the
right-hand side  of the program. Formally, the interval linear program we will be considering is the following (a special case of Type III)
\begin{equation}
\label{ILPb}
[\text {ILP}_b]\;\;\;\;\;\;  \min\; c^Tx \;\; \text{subject to} \;\;Ax \le \mathbf{b},\;\; x\geq 0,
\end{equation}
where $c\in \mathbb{R}^n$, $\mathbf{b}\in \mathbb{IR}^{m}$, and $A\in \mathbb{R}^{m\times n}$ are given. The linear program and an optimal solution (or the set of all optimal solutions), if one exists, associated with a given scenario $b\in\mathbf{b}$ are denoted by LP$(b)$ and $s(b)$, respectively. We also denote by $z(b)$ the optimal value corresponding to LP$(b)$ (infinity and infeasiblity are also allowed). We  focus on solving the two following optimization problems
\begin{align}
\underline{f}&= \{ \min\; f(x)\;\; \text{subject to} \;\; x \in \Omega_b \label{best-b}\},\\
\overline{f}&=\{ \max \;f(x)\;\; \text{subject to} \;\; x \in \Omega_b \label{worst-b}\},
\end{align}
where $\Omega_b$ is the optimal set of ILP$_b$. In the rest of the paper, we will refer to this special case of the outcome range problem as ORP$_b$. 
\begin{remark}
From an application perspective, solving ORP$_b$ is meaningful when the set $\Omega_b$ is not empty and the two values $\underline{f}$ and $\overline{f}$ are finite, i.e., the set  $\Omega_b$ is bounded (see \cite{thesis,garajova2019optimal} for conditions for emptiness and boundedness of $\Omega_b$ ). In what follows, we will assume this is the case. 
\end{remark}
\section {Computational Complexity of the Outcome Range Problem (ORP$_b$)} \label{properties}
We here address the computational complexity of ORP$_b$. Some additional notation is needed at this point. Let us recall the linear program associated with a given $b \in \mathbf{b}$ (i.e., LP$(b)$)
 \begin {equation*}\label{eq:basis}
 \min\; c^T x\;\;\textrm{subject to}\;\; Ax\le b,\; x\ge0.
 \end{equation*}
 The standard form reads 
  \begin{equation}\label{eq:refor}
  \min\; c^T x+0^Td\;\;\textrm{subject to}\;\; Ax+Id=b,\; x\ge0,\; d\ge 0,
  \end{equation}
  where $d\in \mathbb{R}^m$ is the vector of slack variables and $I \in \mathbb{R}^{m \times m}$ is the identity matrix. Let us define $\tilde{A}:=[A|I]$, $\tilde{c}^T:=[c^T|0^T]$, and $\tilde{x}^T:=[x^T|d^T]$. We rewrite (\ref{eq:refor}) as
  \begin{equation}\label{eq:lp}
  \min\; \tilde{c}^T \tilde{x}\;\;\textrm{subject to}\;\; \tilde{A}\tilde{x}=b,\; \tilde{x}\ge0.
\end{equation}
 We similarly define $\tilde{r}^T:=[r^T|0^T]$. 
\begin{definition}
By a basis $B$ we mean an index set $B \subseteq \{1,\dots,n+m\}$ such that $\tilde{A}_B$ is nonsingular, where a subscript $B$ on a matrix (row vector) denotes
the submatrix (subvector) composed of columns indexed by $B$. That is, set $B$ is the set of indices associated with basic variables. Analogously, an index set $N:=\{1,\dots,n+m\} \setminus B$ indicates indices for  nonbasic variables and as a subscript it represents restriction to nonbasic indices. 
 \end{definition}

If the linear program (\ref{eq:lp}) admits a finite optimal value, there exists an optimal basic solution which corresponds to an optimal basis. A basis $B$ is an optimal basis of LP (\ref{eq:lp}) if and only if it satisfies the following conditions
\begin{subequations} \label{eq:basis}
 \begin{align}
     \tilde{A}_B^{-1}b &\ge 0,\\
     \tilde{c}_N^T-\tilde{c}_B^T\tilde{A}^{-1}_B\tilde{A}_N &\ge 0^T.\label{eq:opt}
 \end{align}
 \end{subequations}

Now let us recall the  assumptions under which the optimal set $\Omega_b$ of ILP$_b$ can be explicitly defined. 
\begin{definition} Let a basis $B$ be given. An ILP$_b$ problem is said to be
{\it B-stable}, if $B$ is an optimal basis of LP($b$) for all $b\in \mathbf{b}$. Furthermore, it is called {\it unique B-stable} if it is B-stable
and the optimal basis of LP($b$) is unique for all $b\in\mathbf{b}$.
\end{definition}
\noindent B-stability is a very important property in interval linear programming because it can simplify the description of  the optimal set. In the case of unique B-stability of  ILP$_b$, the optimal set ($\Omega_b$) can  be described by a polyhedral set. 
\begin{lemma}\label{lem:omega} \cite{beeck1978linear}
If (\ref{ILPb}) is {\it unique B-stable} with the optimal basis $B$, the optimal  set ($\Omega_b$) is described by the following linear system \footnote{We adopt Lemma \ref{lem:omega} from the results discussed in \cite{hladik2014determine} (see \cite{hladik2014determine} for more details). }
\begin{equation*}
\tilde{A}_B\tilde{x}_B\leq \overline{b}, \;\;-\tilde{A}_B\tilde{x}_B\le -\underline{b}, \;\; \tilde{x}_B \geq 0,\;\; \tilde{x}_N =0. 
\end{equation*}
\end{lemma}

Another relevant topic in interval linear programming is determining the optimal value range, that is, the problem of finding the best and the worst optimal values among all the optimal values obtained over all data perturbations.
We define the optimal value range of ILP$_b$ (\ref{ILPb}) as
\begin{align}
    \underline{z}&:=\inf\; \{z(b):\; b \in \mathbf{b}\}, \label{eq:zlow1}\\
    \overline{z}&:=\sup\;\{ z(b):\; b \in \mathbf{b}\}\label{eq:zup1}.
\end{align}
Note that (\ref{eq:zlow1}) and (\ref{eq:zup1}) can assume any value, including infinity and infeasibility. The interval $[\underline{z},\overline{z}]$ then gives the optimal value range. By \cite{chinneck2000linear,vajda1961mathematical}, we know that for ILP$_b$ (\ref{ILPb})
\begin{align}
    \underline{z}&=\{\min\; c^Tx\;\; \textrm{subject to}\;\; Ax\le \overline{b},\; x\ge 0\},  \label{eq:zlow2}\\
    \overline{z}&=\{\min\; c^Tx\;\; \textrm{subject to}\;\;  Ax \le \underline{b},\; x \ge 0\}\label{eq:zup2}.
\end{align}

Now we analyze the computational complexity of the outcome range problem. Specifically, Theorem \ref{NP-hard} assesses the computational complexity of ORP$_b$.  Proposition \ref{sovable1} considers a special case of ORP$_b$ which is polynomially solvable. Finally, Proposition \ref{propo:poly2} and Corollary \ref{coro:eq} investigate another polynomially solvable case by exploiting a relation between ORP$_b$ and the optimal value range problem.
\begin{theorem} \label{NP-hard}
Problem ORP$_b$ is NP-hard.
\end{theorem}
\begin{proof}
We proceed by a different interval-related problem which is known to be NP-hard. Let us consider an ILP problem of Type I with a fixed coefficient matrix and a fixed objective vector (i.e., fixed $A$ and $c$), i.e.,
\begin{equation}
\label{type1}
\min\; c^Tx \;\;\text{subject to}\;\;Ax=\mathbf{b},\;\; x\geq 0.
\end{equation}
Let $\Xi$ be the optimal set of (\ref{type1}). By Theorem 7 in \cite{garajova2019optimal} (p. 282), we know that computing the exact interval hull of $\Xi$ is NP-hard. Now let us reformulate problem (\ref{type1}) as follows
\begin{equation}
\label{reform}
 \min\; c^Tx \;\;\text{subject to}\;\;Ax\le\mathbf{b},\;\;-Ax\le-\mathbf{b},\;\; x\geq 0.
\end{equation}
We know by Theorem 2 in \cite{garajova2018interval} (p. 606) that the optimal set of (\ref{reform}) is equal to the optimal set of (\ref{type1}). For the sake of  simplicity, let us introduce the following notation

$$A^{\prime}:=
\left[
\begin{array}{c}
A\\
\hline
-A
\end{array}
\right],\;\;\mathbf{b}^{\prime}:=
\left[
\begin{array}{c}
\mathbf{b}\\
\hline
-\mathbf{b}
\end{array}
\right].$$
We then can rewrite the problem (\ref{reform}) as an ILP$_b$, that is,
$$
 \min \;c^Tx \;\;\text{subject to}\;\;A^{\prime}x\leq\mathbf{b}^{\prime},\;\; x\geq 0.
$$
Therefore, we can conclude that $\Xi=\Omega_b$. As a result, we can say that computing the exact interval hull of $\Omega_b$ is also NP-hard.  Now if we consider $f(x)=x_i$, for any $i\in \{1,\dots,n\}$, we can conclude that ORP$_b$ is NP-hard.
\end{proof}

\begin{proposition} \label{sovable1} \label{B_L}
If ILP$_b$ is unique B-stable, then ORP$_b$ is polynomially solvable.
\end{proposition}
\begin{proof}
Let the basis $B$ be the unique optimal basis for all the data realizations, then based on Lemma \ref{lem:omega}, ORP$_b$ is equivalent to solving the two following linear programs
\begin{align}
\underline{f}&=\{\min\;\tilde{r}^T_B\tilde{x}_B\;\;\text{subject to}\;\;\tilde{A}_B\tilde{x}_B\leq \overline{b}, \;\;-\tilde{A}_B\tilde{x}_B\le -\underline{b}, \;\; \tilde{x}_B \geq 0,\;\; \tilde{x}_N =0\}, \label{eq:bstabl}\\
\overline{f}&=\{\max\;\tilde{r}^T_B\tilde{x}_B\;\;\text{subject to}\;\;\tilde{A}_B\tilde{x}_B\leq \overline{b}, \;\;-\tilde{A}_B\tilde{x}_B\le -\underline{b}, \;\; \tilde{x}_B \geq 0,\;\; \tilde{x}_N =0\}.\label{eq:bstabu}
\end{align}	
\end{proof}

Let $\mathscr{B}$ be the set of all optimal bases of ILP$_b$ (if any).
Let us consider a given $B \in \mathscr{B}$. We can associate with it, by (\ref{eq:opt}), a cone containing all the cost vectors that are optimal for $B$, that is, $H^B:=\{\psi \in \mathbb{R}^{m+n}:\; \psi_{N}^{ T}-\psi_{B}^{T}\tilde{A}_{B}^{-1}\tilde{A}_{N} \ge 0^T\}$. We define $\mathscr{C}$ as the intersection of all the cones containing all the cost vectors that are optimal for all the optimal  bases, i.e.,
$$\mathscr{C}=\bigcap_{B \in \mathscr{B}} H^B.$$
The following proposition states another polynomially solvable case of ORP$_b$ by leveraging a relation with the optimal value range problem.

\begin{proposition}
\label{propo:poly2}
Suppose that $\underline{z}$ and $\overline{z}$ are finite values. If $r$ is such that $r\in \mathcal{C}$, then ORP$_b$ is polynomially solvable.
\end{proposition}
\begin{proof}
Let us recall that the optimal value range of ILP$_b$ is polynomially solvable, that is,
$$
[\text{P}_1]: \;\underline {z}=\{\min\; c^T x\;\;\text {subject to}\;\;\; Ax\leq\overline{b},\;x\geq0\},\quad \quad [\text{P}_2]:\; \overline{z}=\{\min\; c^T x\;\;\text {subject to}\;\;\; Ax\leq\underline{b},\;x\geq0\},
$$
 and that ORP$_b$ consists in solving the following two optimization problems
$$
[\text{P}_3]:\; \underline {f}=\{\min\; r^T x\;\;\text {subject to}\;\;\; x\in \Omega_b\},\quad \quad [\text{P}_4]:\; \overline{f}=\{\max\; r^T x\;\;\text{subject to}\;\;\; x\in \Omega_b\}.
$$
From the hypothesis, we know that $\underline{z}$ is a finite value. Let $x^{*}$ be an optimal solution of P$_1$, i.e., $\underline{z} = c^Tx^{*}$. By definition, we know that $x^{*}\in \Omega_b$. Since $r\in \mathcal{C}$, we can write
$$
 r^Tx^*=\{\min\;\ r^T x\;\;\text {subject to}\;\;\; Ax\leq\overline{b},\;x\geq0\}.
$$ 
Let us now consider a generic  $\hat{x} \in \Omega_b$ in P$_3$, which is an optimal solution of the linear program associated with a scenario $\hat{b} \in \mathbf{b}$. Again, given $r\in \mathcal{C}$, we have
$$
 r^T\hat{x}=\{\min\;\ r^T x\;\;\text {subject to}\;\;\; Ax\leq\hat{b},\;x\geq0\}.
$$
Since $\hat{b}\le \overline{b}$ and $A\hat{x}\le \overline{b}$, we can say $r^T\hat{x}\geq r^Tx^{*}$. This is true for any vector $\hat{x} \in \Omega_b$, and thus $x^{*} $ is also an optimal solution to P$_3$. Therefore, we can compute $\underline{f}$ by
$$
\underline {f}=\{\min\; r^T x\;\;\text {subject to}\;\;\; Ax\leq\overline{b},\;x\geq0\},
$$
which is polynomially solvable. We can use a similar argument for P$_2$ and P$_4$.
\end{proof}
Now it is easy to see that the following special case of Proposition \ref{propo:poly2} holds. 
\begin{corollary}\label{coro:eq} 
Suppose that $f(x)=c^Tx$. If $\underline{z}$ and $\overline{z}$ are finite values, then we have $[\underline{f},\overline{f}]=[\underline{z},\overline{z}]$.
\end{corollary}
Note that even if we assume that $f(x)=c^Tx$, the outcome range problem is not equivalent to the optimal value range problem in general. Below, we illustrate this by an example.
\begin{exmp}\label{example:counter1}
Consider the following ILP$_b$ problem
$$\min \; -4x_2\;\;\text{subject to}\;\;x_1+x_2\le[-1,5],\;x_1,x_2\ge0,$$
and  let $f(x)=-4x_2$ also be an outcome function. By (\ref{eq:zup2}), it is easy to see that $\overline{z}$ is infeasible, and by applying (\ref{eq:zlow2}) we get $\underline{z}=-20$. Therefore, $[-20,\infty]$ gives the optimal value range.\footnote{We denote infeasibility by the convention $\min\;\emptyset=\infty$.} 
\begin{figure} [h!] 
\centering
    \includegraphics[scale=1]{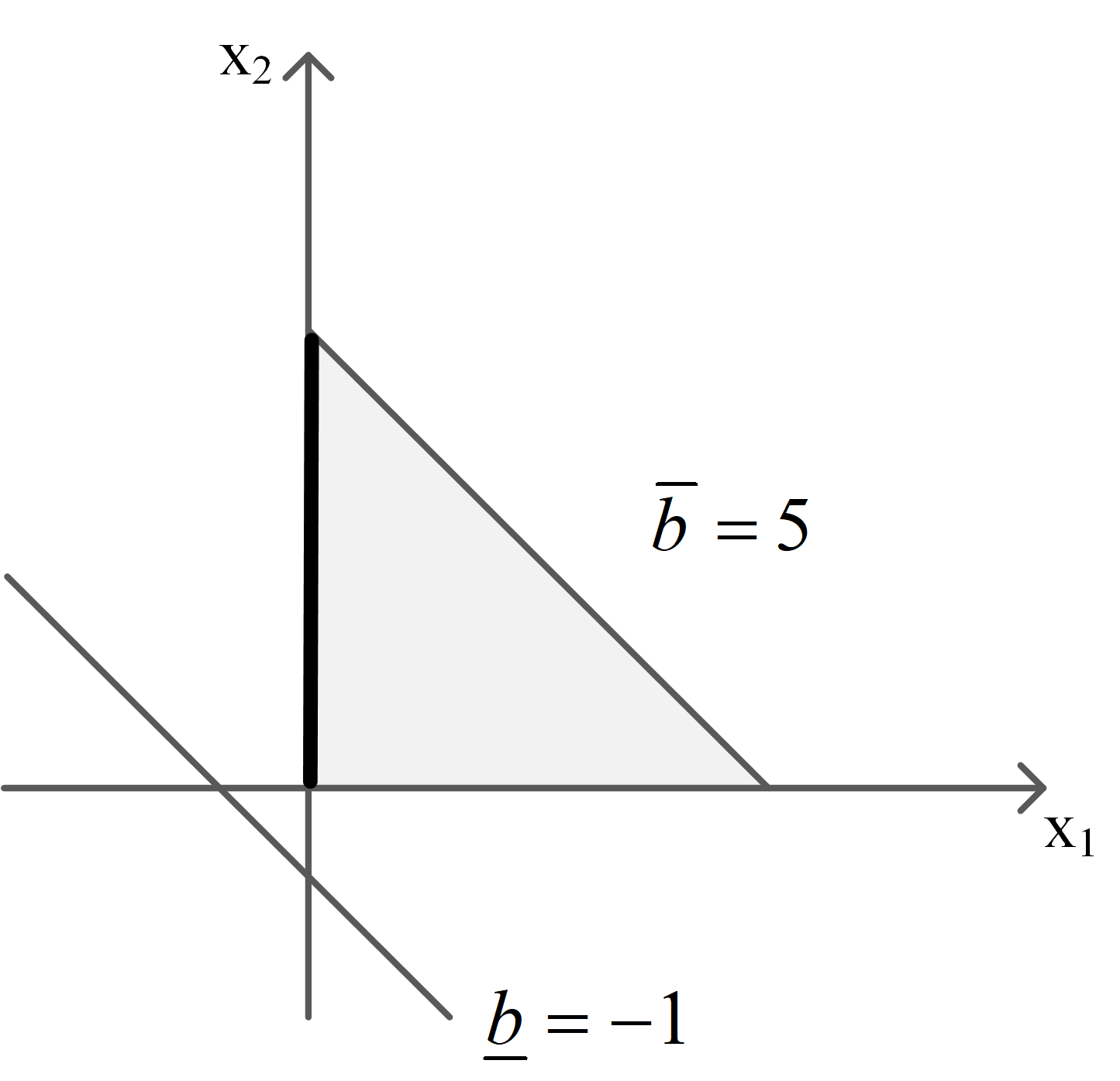}
\caption{(Example \ref{example:counter1}) Union of all feasible sets in light gray; set of all optimal solutions in bold.}
\label{fig:counter1}
\end{figure}
However, from Figure \ref{fig:counter1}, it is not hard to observe that $\underline{f}=-20$ (scenario $b=5$) and $\overline{f}=0$ (scenario $b=0$). Hence, the outcome function $f(x)$ ranges in the interval $[-20,0]$, which is different from the optimal value range.  
\end{exmp}
Corollary \ref{coro:eq} and Example \ref{example:counter1} indeed imply that the outcome range problem can be seen as a generalized form of a special case of the optimal value range problem \cite{hladik2018worst}.
\section{Properties of the Outcome Range Problem (ORP$_b$)}\label{sec:prop}
In this section, we study some theoretical properties of ORP$_b$ aimed at characterizing the scenarios corresponding to the optimal values of (\ref{best-b}) and (\ref{worst-b}). Throughout this section we only report results related to the computation of $\underline{f}$. All the results are applicable to the computation of $\overline{f}$ as well. Let us introduce some definitions first.
\begin{definition} \label{def:pos}
 A given scenario $b \in \mathbf{b}$ is referred to as
\begin{enumerate}[label=(\roman*)]
\item \textit{a middle scenario if}
$$\underline{b}_i < b_i < \overline{b}_i,\;\forall i\in\{1,\dots,m\}.$$
\item \textit{a weakly extremal scenario if}  
$$ b_i=\underline{b}_i \; \vee\; b_i=\overline{b}_i,\; \textrm{for some}\; i\in\{1,\dots,m\}.$$
\item \textit{a strongly extremal scenario if} 
$$ b_i=\underline{b}_i \; \vee\; b_i=\overline{b}_i,\; \forall i\in\{1,\dots,m\}.$$
\end{enumerate}
\end{definition}
\noindent Note that, according to the above definition, a strongly extremal scenario is also a weakly extremal scenario, but the opposite does not hold true. From the geometrical standpoint, given a hypercube $\mathbf{b}$, a middle scenario is in the interior of the hypercube, a weakly extremal scenario is on the boundary of the hypercube, and a strongly extremal scenario is a vertex of the hypercube.
\begin{definition}
$b^* \in \mathbf{b}$ is  an {\it optimal scenario} of (\ref{best-b}) if $\underline{f}=f(x^*)$, where $x^* \in s(b^*)$.
\end{definition}
\begin{definition}\label{def:global}
Given an optimal scenario $b^*$ for (\ref{best-b}), an  optimal basis $B^*$ of the linear program LP$(b^*)$ is a {\it global optimal basis} of (\ref{best-b}).
\end{definition} 
\begin{remark}\label{rem:glob}
Note that, given a global optimal basis $B^*$ of (\ref{best-b}), the optimal value $\underline{f}$ and the optimal scenario $b^*$ are the optimal value and an optimal solution, respectively,  of the following linear program
\begin{equation*}
    \min\; \tilde{r}^T_{B^*}\tilde{A}^{-1}_{B^*}b\;\;\text{subject to}\;\; \tilde{A}^{-1}_{B^*}b \ge 0,\; b\in \mathbf{b},
\end{equation*}
in variables $b$.
\end{remark}

The results to follow identify conditions to characterize the optimal scenario $b^*$ either as a middle or a weakly (strongly) extremal scenario. 
\begin{proposition}\label{prop:weak}
If $(0,\dots,0)^T \notin \mathbf{b}$, then there exists a weakly extremal scenario $\hat{b}$ such that $b^*=\hat{b}$.
\end{proposition}
\begin{proof}
Let the basis $B^*$ be a global optimal basis of (\ref{best-b}). By Remark \ref{rem:glob}, the optimal scenario $b^*$ is an optimal solution of the following linear program  
\begin{equation} \label{eq:bs}
\min\; \tilde{r}^T_{B^*}\tilde{A}^{-1}_{B^*}b\;\;\text{subject to}\;\; \tilde{A}^{-1}_{B^*}b \ge 0,\; b\in \mathbf{b}.
\end{equation}
Let us refer to the feasible set of (\ref{eq:bs}) as $P_{B^*}$. It is known that there exists an optimal solution of (\ref{eq:bs}) which is an extreme point of $P_{B^*}$.  

We know that each vertex of $P_{B^*}$ is a vector satisfying all the constraints such that at least $m$ of the constraints are binding and are linearly independent. We also know that matrix $\tilde{A}^{-1}_{B^*}$ is a full rank square matrix of order $m$. Note that since $(0,\dots,0)^T \notin \mathbf{b}$, then $\tilde{A}^{-1}_{B^*}b \neq 0$ for each $b \in \mathbf{b}$. Therefore, any extreme point of $P_{B^*}$ corresponds to a vector for which at least one of the constraints in $b\in \mathbf{b}$ is binding. We can then conclude that there exists $i \in \{1,\dots,m\}$ such that $b^*_i=\underline{b}_i$ or $b^*_i=\overline{b}_i$. This completes the proof.
\end{proof}

\noindent Note that $b^*$ can still be a weakly extremal scenario even in the case of $(0,\dots,0)^T\in \mathbf{b}$, but this requires the vector $(0,\dots,0)^T$ not to be a middle scenario.
\begin{corollary}
Suppose  that vector $(0,\dots,0)^T$ is such that it is  a weakly extremal scenario of of the interval vector $\mathbf{b}$. Then there exists a weakly extremal scenario $\hat{b}$ such that $b^*=\hat{b}$.
\end{corollary}

\noindent Proposition \ref{prop:weak} also reveals  an interesting observation related to middle scenarios.
\begin{corollary}\label{cor:mid}
If the optimal scenario $b^*$ is unique and it is a middle scenario, then we have $b^*=(0,\dots,0)^T$.
\end{corollary}
\begin{proof}
Similar to the proof of Proposition \ref{prop:weak}, let the basis $B^*$ be a global optimal basis of (\ref{best-b}). Consequently, the optimal scenario $b^*$ is an optimal solution of (\ref{eq:bs}). Suppose for the sake of contradiction that $b^*$ is unique and it is a middle scenario, i.e., $\underline{b}_i<b^*_i<\overline{b}_i$ for all $i\in\{1,\dots,m\}$, but $b^*\ne (0,\dots,0)^T$. Since $b^*$ is the unique optimal solution of (\ref{eq:bs}),  then $m$ linearly independent constraints needs to be binding on $b^*$ to form an extreme point. $\tilde{A}_{B^*}^{-1}$ is a full rank square matrix of order $m$; thus, we need to have $\tilde{A}_{B^*}^{ -1}b=0$. This system possesses one unique solution which is $(0,\dots,0)^T$. Therefore, $b^*=(0,\dots,0)^T$. We derive a contradiction here, and this completes the proof.
\end{proof} 
Note that the opposite of Corollary \ref{cor:mid} is not valid. That is, if the optimal scenario is $b^*=(0,\dots,0)^T$ , then this does not necessarily imply neither that $b^*$ is unique, nor that it is a middle scenario. The following provides a counterexample.
\begin{exmp}\label{example:counter2}
Consider the following interval linear program
$$\min\; 5x_1+6x_2\;\;\text{subject to}\;\; 4x_1+5x_2\le [0,5],\;x_1,x_2\ge 0,$$
\begin{figure}
\centering
    \includegraphics[scale=1]{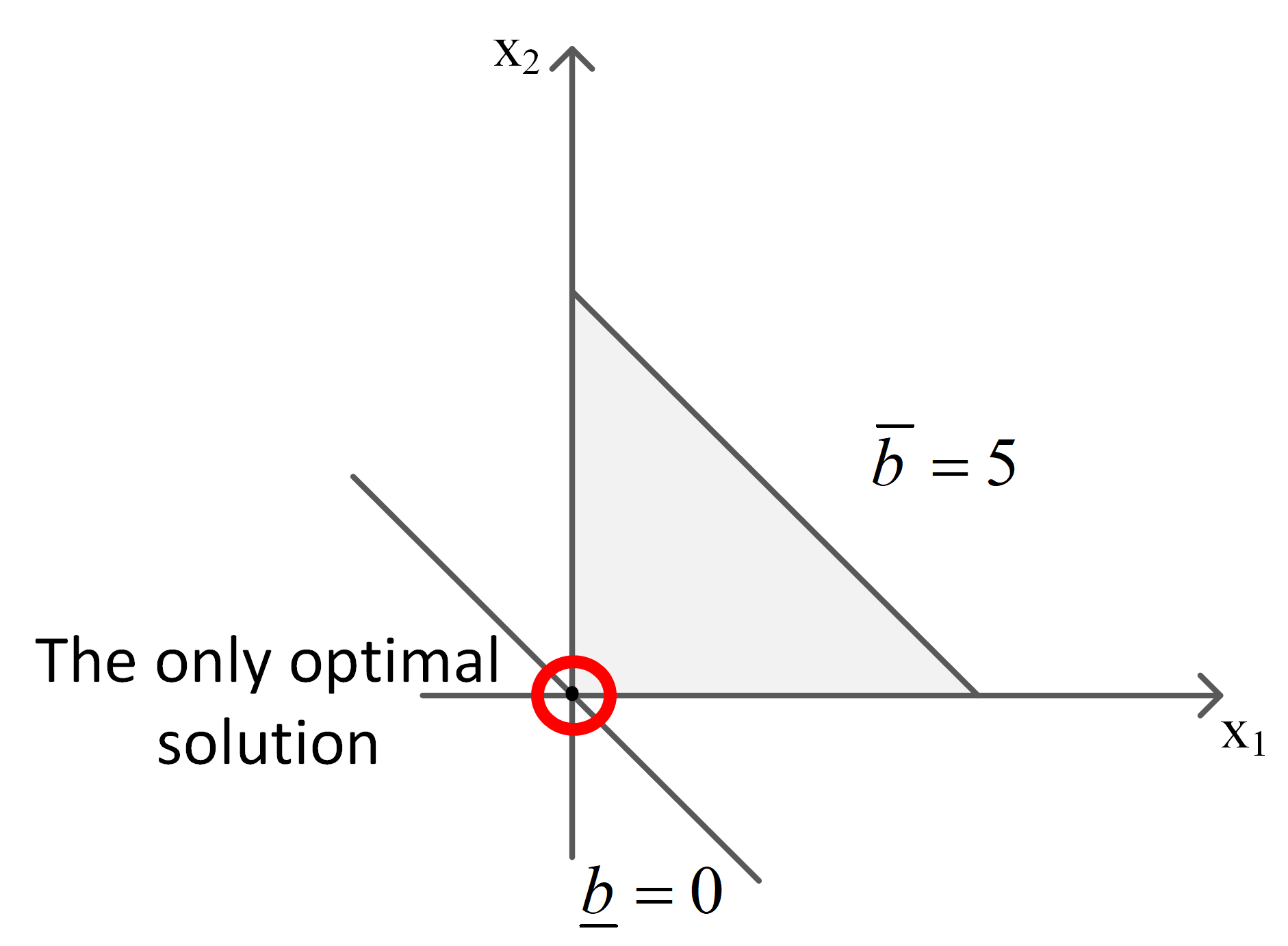}
\caption{(Example \ref{example:counter2}) Union of all feasible sets in light gray; the only optimal solution is red circled.}
\label{fig:counter2}
\end{figure}
and let $f(x)=10x_1+3x_2$ be an outcome function. From Figure \ref{fig:counter2}, we observe that  $x^*=(0,0)$  is the unique optimal solution for all the linear programs LP$(b)$, for all $b \in \mathbf{b}$, that is, $\Omega_b=\{(0,0)\}$. Hence, we have $\underline{f}=0$. It is not hard to see that any scenario in the interval $[0,5]$ is an optimal scenario for ORP$_b$. Therefore, $b^*=0$ is an optimal scenario, but it is neither unique   nor a middle scenario.
\end{exmp}
We now present a condition under which $b^*$ is a strongly extremal scenario.
\begin{proposition}
If $B^*$ is non-degenerate for LP$(b^*)$, then $b^*$ is a strongly extremal scenario.
\end{proposition}
\begin{proof}
Let us recall that, given a global optimal basis $B^*$, the following linear program returns $\underline{f}$ and $b^*$.
\begin{equation*} 
\min\; \tilde{r}^T_{B^*}\tilde{A}^{-1}_{B^*}b\;\;\text{subject to}\;\; \tilde{A}^{-1}_{B^*}b \ge 0,\; b\in \mathbf{b}
\end{equation*}
We know that $B^*$ is a non-degenerate  optimal basis of LP$(b^*)$, and thus $b^*$ is such that $\tilde{A}^{-1}_{B^*}b^*> 0$. Therefore, to have an extreme point, $m$ linearly independent  constraints in $b\in \mathbf{b}$ need to binding on $b^*$, that is, $ b^*_i=\underline{b}_i$ or $b^*_i=\overline{b}_i$ for all $i\in\{1,\dots,m\}$. The proof is now concluded.
\end{proof} 
\noindent Finally, the following observation states another case under which an optimal scenario $b^*$ is strongly extremal. It follows directly from  Proposition \ref{propo:poly2} in Section \ref{properties}.

\begin{observation}
Assume that $\underline{z}$ is a finite value. If $r$ is such that $r\in \mathcal{C}$, then we have $b^*=\overline{b}$, which is a strongly extremal scenario.
\end{observation}

\section{Solution Methods} \label{techniques}
In Section \ref{properties}, we show that ORP$_b$ is an NP-hard problem in general; however, when the underlying ILP$_b$ is unique B-stable, we can solve  ORP$_b$ to optimality in polynomial time. B-stability is unlikely to occur when we are dealing with wide intervals and large problems. Therefore, unless P$=$NP, there is no hope for any polynomial-time solvable characterization of the problem in general. As such,  we here describe two different approaches to approximate the optimal solution of   ORP$_b$. Specifically, we present a super-set based method and a local search algorithm to approximate (\ref{best-b}) and (\ref{worst-b}).

\subsection{Super-set based Method}
\noindent As stated in the previous sections, an explicit description of the optimal set $\Omega_b$ is not always available. However, if we are able to find a super-set $E(\Omega_b)$  containing it, i.e., such that $\Omega \subseteq E(\Omega_b)$,  we could then approximate the  optimal values $\underline{f}$ and $\overline{f}$ by solving the two following optimization problems
\begin{align}
\underline{f}^{L}& =\{\min\; r^Tx\;\ \text{subject to}\;\; x\in E(\Omega_b)\}, \label{best-encl}\\
\overline{f}^{U}&= \{\max\; r^Tx\;\ \text{subject to}\;\; x\in E(\Omega_b)\},  \label{worst-encl}
\end{align}
 where  $\underline{f}^{L}$ and $\overline{f}^{U}$ denote a lower bound of $\underline{f}$ and an upper bound of $\overline{f}$, respectively.

To define a super-set $E(\Omega_b)$, we can apply some duality properties in linear programming. More specifically, let us recall the dual of  ILP$_b$ for a particular $b\in\mathbf{b}$ (i.e., LP $(b)$),
\begin{equation*}
 \max \; b^Ty \;\;\text{subject to}\;\;A^Ty \leq c,\;y\leq0,
\end{equation*}
where $y\in\mathbb{R}^m$ is the vector of decision variables. By the strong duality condition in linear programming, we can describe the optimal solution set of LP$(b)$ by means of
  the following linear system
\begin{equation*} 
Ax \leq b, \;\;x\geq 0, \;\;\; A^Ty \leq c, \;\;y\leq0,\;\; c^Tx = b^Ty.
\end{equation*}
Let us assume, without loss of generality, that $b$ is a vector of decision variables varying within the interval vector $\mathbf{b}$. We can then characterize  the optimal set $\Omega_b$ as
\begin{equation} \label{eq:nonlin}
Ax \leq b, \;\;x\geq 0, \;\;\; A^Ty \leq c, \;\;y\leq0,\;\; c^Tx = b^Ty,\;\;b\in\mathbf{b},
\end{equation}
in variables $x,y,b$. This leads to a nonlinear programming problem, due to the nonlinear term $ b^Ty$, which is very difficult to solve. Therefore, we linearize it by using McCormick envelope techniques \cite{McCormick1976}. Let $[\underline{y},\overline{y}]$ be an interval enclosure for $y$. We then apply overestimator and underestimator constraints  to linearize the nonlinear constraint $c^Tx = b^Ty$. The resulting system reads
\begin{subequations}  \label{eq:mc}
\begin{align}
Ax \leq b, \;\;x\geq 0, \;\;\; A^Ty \leq c, \;\;y\leq0,\;\;b\in\mathbf{b}, \\
c^Tx\le \underline{y}^Tb+\overline{b}^Ty-\overline{b}^T\underline{y}, \label{over1}\\
c^Tx \le\overline{y}^Tb+\underline{b}^Ty-\underline{b}^T\overline{y}, \label{over2}\\
c^Tx \ge \overline{y}^Tb+\overline{b}^Ty-\overline{b}^T\overline{y}, \label{under1}\\
c^Tx \ge \underline{y}^Tb+\underline{b}^Ty-\underline{b}^T\underline{y}, \label{under2}
\end{align}
\end{subequations}
where (\ref{over1})-(\ref{over2}) are called overestimators, while (\ref{under1})-(\ref{under2}) are called underestimators.

 System (\ref{eq:mc}) is a super-set containing $\Omega_b$. Therefore, we can use it to solve problems (\ref{best-encl}) and (\ref{worst-encl}). To compute an interval enclosure $[\underline{y},\overline{y}]$ for $y$, we can apply the contractor algorithm in \cite{contractor}. Briefly, the contractor algorithm is an iterative refinement algorithm. It starts with an enclosure of an optimal set and contracts such an enclosure at each iteration until improvement is insignificant. It runs in polynomial time, and it returns a sufficiently tight interval enclosure for $y$. We use this algorithm in our experiment in Section \ref{results} to get the interval enclosure $[\underline{y},\overline{y}]$.

\subsection{Local Search Algorithm}
In this section, we describe a local search algorithm to approximate $\underline{f}$ and $\overline{f}$. Local search is a heuristic method which, given a current feasible solution, tries to improve it by exploring feasible solutions in its neighborhood \cite{siarry2016metaheuristics}. Since  the returned solution will be a member of $\Omega_b$, the local search algorithm gives a lower bound for $\overline{f}$ (denoted as $\overline{f}^L$)  and an upper bound for $\underline{f}$ (denoted as $\underline{f}^U$). Our algorithm starts with an initial solution associated with a given scenario $b \in \mathbf{b}$, then it explores two neighborhoods of the solution, obtained by perturbing $b$, to find a new solution. If the new solution is better than the current one, then it stores the solution and starts a new iteration. The algorithm proceeds in this way until a stopping condition is met. 
\noindent We discuss our neighborhood structure and details of our algorithm next.
\subsubsection{Neighborhood Structure}
\noindent We define our neighborhood structure in the scenario space, that is, given an optimal solution of a linear program associated with a particular scenario $b \in \mathbf{b}$, we define two neighborhood structures, namely plus and minus neighborhoods, obtained by perturbing $b$. Specifically, a plus neighbor (minus neighbor) of a scenario $b \in \mathbf{b}$ is obtained by increasing (decreasing) some components of $b$ by a given quantity. The number of components of $b$ to be perturbed and the amount of perturbation (increment or decrement) are adjustable values. We formally define our neighborhood structures as
\begin{align} 
N^{+}_{k,h}(b)&:=\{ \tilde{b}\in \mathbf{b}: \tilde{b}_i=b_i+k\phi^+_i,\; \tilde{b}_j=b_j,\; i\in P,\; j \neq i, \; P\in P(h) \},\label{plus}\\
N^{-}_{k,h}(b)&:=\{ \tilde{b}\in \mathbf{b}: \tilde{b}_i=b_i-k\phi^-_i,\; \tilde{b}_j=b_j,\; i\in P,\; j \neq i,\; P \in P(h)\}, \label{minus}
\end{align}
where $\phi_i^{+}$ and $\phi_i^{-}$ represent the maximum allowable perturbation of $b_i$, and they are computed, respectively, as $\phi^+_i=\overline{b}_i-b_i$ and $\phi^-_i=b_i-\underline{b}_i$. Both the plus and the minus neighborhoods of a given scenario are defined depending on two parameters: parameter $k \in (0,1]$ which is a fraction of $\phi_i^{\pm}$ by which we perturb $b_i$, and parameter $h \in (0,1]$ which is a fraction of the total number of components in vector $b$ which we  perturb simultaneously. Let us consider the set $\{1,\ldots,m\}$ as the index set of components in vector $b$. With each value of the parameter $h$, we denote by $P(h)$ the collection of all the possible subsets of $\{1,\ldots,m\}$ of cardinality $\left \lfloor{h\times m}\right \rfloor $ \footnote{$\lfloor. \rfloor$ denotes the floor function.}, that is, $P(h):=\{P \subseteq \{1,\dots,m\}: |P|=\left \lfloor{h\times m}\right \rfloor\}$. Basically, each subset P in $P(h)$ represents a choice of $\left \lfloor{h\times m}\right \rfloor $ components  of a current scenario $b \in \mathbf{b}$, which are simultaneously perturbed. Given a scenario $b \in \mathbf{b}$ and a value of $h$, the number of neighbors in either $N^{+}_{k,h}(b)$ or $N^{-}_{k,h}(b)$ is equal to ${m}\choose{\lfloor{h\times m} \rfloor}$. Finally, for a particular $b\in \mathbf{b}$, a fixed value of $k$,  a fixed value of $h$, and  a set $P \in P(h)$, we determine a neighbor $\tilde{b}$ in either $N^{+}_{k,h}(b)$ or $N^{-}_{k,h}(b)$ ,  and denote by $f^{+}_{b,k,h,P}$ ($f^+$ for short when no confusion arises) or  $f^{-}_{b,k,h,P}$ ($f^-$ for short when no confusion arises) the value of the outcome function computed on an optimal solution of the linear program associated with $\tilde{b}$.
\begin{algorithm}[]
\small             
\caption{Local search algorithm to compute $\underline{f}^U$}     
\label{alg1}     
\SetAlgoLined
\small
\KwSty{Input:} { $A$, $\mathbf{b}$, $c$, $r$, $Q$, $V$, max{\text -}shakes, threshold}

\KwResult{the best approximation among all approximations}

Compute $\underline{f}^U_{int}$.

Set $\underline{f}^U \leftarrow \underline{f}^U_{int}$ and $b \leftarrow \hat{b}_{int}$.

Put $q \leftarrow 1$ and $v \leftarrow 1$.

Set $k \leftarrow Q^{(q)}$ and $h\leftarrow V^{(v)}$.

Randomly select a set $P$ in $P(h)$.

Put $u\leftarrow 0$ and $o \leftarrow 1$.

\While{{$u \leq max{\text -}shakes$}}{
Compute $f^{+}$ and $f^{-}$.

Set $\hat{f}\leftarrow \min \{f^{+},f^{-}\}$ and let $\hat{b}$ be the corresponding right-hand side.

Determine $improvement \leftarrow \underline{f}^U-\hat{f}$.

\uIf {$improvement\geq threshold$} 
{Set $\underline{f}^U \leftarrow \hat{f}$ and $b \leftarrow \hat{b}$.

Set $k \leftarrow Q^{(1)}$.}

\uElseIf {$q< |Q|$}{
Set $q\leftarrow q+1$.

Put $k\leftarrow Q^{(q)}$.}

\uElseIf {$v< |V|$}{
\uIf {$o\le \lfloor\frac{1}{h} \rfloor$}{

Set $o\leftarrow o+1$.

Let $\Gamma$ be the set of all indices chosen so far for the current $h$.

Randomly generate a set $P$ in $P(h)$ such that $P\cap \Gamma=\emptyset$.

Set $k \leftarrow Q^{(1)}$.}

\Else{Update $v\leftarrow v+1$ and $o\leftarrow 1.$

Put $h\leftarrow V^{(v)}$.

 Randomly generate a set $P$ in $P(h)$.

Set $k\leftarrow Q^{(1)}$.}
}

\Else{ Update $u\leftarrow u+1$.

 Set $k$, $h$ to their initial values and set  $o$,$q$, and $v$ to 1.

Randomly generate a scenario $b\in \mathbf{b}$.

Randomly generate a set $P$ in $P(h)$.}
}

\end{algorithm}
\subsubsection{The Algorithm}
The pseudo-code Algorithm \ref{alg1}  shows details of our algorithm to compute $\underline{f}^U$; we can apply a similar scheme to compute $\overline{f}^L$. Line 1 contains input of the algorithm: $A$, $\mathbf{b}$, $c$  are parameters of the ILP$_b$, $r$ is the coefficient vector of an outcome function, $Q$  is an ordered set of all the selected values $k$, $V$  is the ordered set of all the selected values  $h$, max-shakes  indicates the stopping condition, and threshold  represents the minimum acceptable improvement during execution of the algorithm. We denote by $Q^{(q)}$ and $V^{(v)}$ the $q$-th and $v$-th elements in the two ordered sets, respectively.

Line 2 computes an initial solution $\underline{f}_{int}^U$, and stores the associated scenario $\hat{b}_{int}$. An initial solution can be computed by solving ORP$_b$ for a randomly generated scenario. The algorithm repeatedly refines an initial solution by using the neighborhood structures defined earlier. 
Lines 4-5 set the initial values of parameters $k$ and $h$. Line 6 generates a set $P$ in the collection $P(h)$. 
Line 7 initiates counter variables. Line 8 checks whether the stopping condition is met. At each iteration, using the neighborhood structures (\ref{plus}) and (\ref{minus}), lines 9-11 determine a potential incumbent solution, and compute the improvement. If the improvement is acceptable (line 12), lines 13-14 update  $\underline{f}^U$ and $b$ , and reset $k$ to its initial value. Otherwise, the algorithm tries the next value of $k$ in $Q$ (lines 15-17). After trying all $k \in Q $, the algorithm chooses a different set $P$ for the current value of $h$, or it tries different values of $h$ (lines 18-29). Specifically, it first randomly selects a new set $P$  in $P(h)$. Note that lines 19-23 generate different sets $P$ so that they are mutually exclusive. After trying a maximum number $\lfloor\frac{1}{h} \rfloor$ of different sets $P$ in $P(h)$  for a given $h$, if still no improvement is achieved, then lines 24-29 choose a new value of $h$ in $V$. The algorithm continues  in this way until all $h$ in $V$ have been selected. Lines 30-35 apply a shaking step. The aim of this step is to move the search to a different area of the search space. After trying all values of $h$ and $k$ without getting  any improvement, a shaking phase starts. In this phase, the input parameters $k$ and $h$  are set to their initial values, counters are re-initialized, a random scenario in $\mathbf{b}$ is generated, and  a set $P$ in $P(h)$ is randomly generated. The algorithm proceeds in this way until the stopping condition is met. Finally, it returns the best approximation among all.

\section {Experimentation} \label{results}
Here, we present our computational experiments and related results to evaluate the performance of our approaches. Since there exists no algorithm in the literature to compare our approaches with, then, in addition to our super-set based method and our local search (LS) algorithm, we also use FMINCON, a nonlinear programming solver in MATLAB, to solve the nonlinear formulation of the ORP$_b$, that is, minimizing (maximizing) $f(x)$ subject to system (\ref{eq:nonlin}). We compare all the methods on two sets of randomly generated instances. The first set, referred to as class 1, is a collection of unique B-stable instances so that the output of our approaches can be compared to the optimal values of the problem (see Proposition \ref{sovable1}). The second set of instances, referred to as class 2, is a series of general instances for which the unique B-stability property is not guarantied. Thus, for this set of instances, the optimal values are not known. 
\subsection {Description of Instances}
We generated class 1 instances using the following procedure. First, for a given problem size ($m,n$) and uncertainty parameter (i.e., interval width) ($\delta$), entries of matrix $A\in \mathbb{Z}^{m\times n}$ were randomly generated in $[-10,10]$ using uniformly distributed pseudorandom integers. Similarly, vectors $c\in \mathbb{Z}^n,\underline{b}\in \mathbb{Z}^m,r\in \mathbb{Z}^n$ were randomly taken in $[-20,-1],[10,20]$, and $[-20,20]$, respectively. Vector $\overline{b}$  was constructed as $\overline{b}=\underline{b}+\delta e$, where $e=(1,...,1)^T$ is a vector of ones with the convenient dimension. To ensure boundedness of the optimal set, we kept entries of the last row of matrix $A$ positive. To have a unique B-stable instance, we found an optimal basis by solving the linear program associated with a randomly chosen scenario, and checked whether the optimal basis is unique and common to all scenarios, i.e., we checked the following conditions\footnote{Here, we adopt the unique B-stability conditions for our problem. See \cite{hladik2014determine} for a thorough investigation of B-stability in interval linear programming.}
\begin{align} 
\tilde{c}_N^T-\tilde{c}_B^T\tilde{A}_B^{-1}\tilde{A}_N> 0^T, \nonumber \\
\tilde{A}_B^{-1}b_c-|\tilde{A}_B^{-1}|b_{\Delta}\geq 0,\nonumber
\end{align}
where $b_c := \frac{1}{2}(\overline{b}+\underline{b})$ and $b_{\Delta}:=\frac{1}{2}(\overline{b}-\underline{b})$ denote the center and the radius of the interval $\mathbf{b}$. If both conditions held true, we saved the instance. Otherwise, we started over the process to generate a new instance. In our experimental study, for class 1 instances, we considered the following problem sizes and values for the uncertainty parameter:  $m=\{10, 30, 50, 80, 100\}$, $n=\{15, 45, 75, 120, 150\}$ and $\delta=\{0.1, 0.25, 0.5, 0.75, 1\}$. We studied 25 different combinations of $m,n,\delta$, and generated 30 instances for each combination, for a total of 750 instances.

We used a similar procedure to generate class 2 instances, except that the unique B-stability was not required for these instances. For class 2 instances, we considered the following problem sizes and values for the uncertainty parameter: $m=\{10, 30, 50, 80, 100,200,300,400,500\}$,\;$n=\{15, 45, 75, 120,150,300,400,\\500,600\}$ and $\delta=\{0.1, 0.25, 0.5, 0.75, 1\}$. We examined 45 different combinations of $m,n,\delta$, and again we generated 30 instances for each combination, for a total of 1,350 instances. 
\subsection{Implementation of Algorithms}
The input parameters for the local search algorithm were chosen as follows. The two ordered sets $Q$ and $V$ were such that $Q=\{0.1,0.25, 0.5, 0.75, 1\}$ and $V=\{0.05,0.1,0.15,0.2, 0.25, 0.3,0.5,1\}$. The max-shake parameter was set equal to one, and the threshold parameter was set equal to 0.001.  FMINCON has five stopping criteria namely maximum iterations, maximum function evaluations, step tolerance, function tolerance, and constraint tolerance. We set the maximum iterations and the maximum function evaluations to 300,000, step tolerance to ($1.000E-10$), and function  and constraints tolerances to ($1.000E-6$). For each problem, we first solved a linear program associated with a randomly generated scenario, and we then took an optimal solution of the linear program as the starting point for the FMINCON. We imposed a time limit of 30 minutes on the solver  for each instance such that if the solver cannot  normally converge to a solution within 30 minutes, it is terminated and its current solution is returned (if it lies within the feasibility tolerance). For the cases the solver reached one of its  internal stopping criteria before reaching the time limit, it started over from a different starting point and continued in this way until either it converged to a solution or it reached the time limit. For the cases for which  the solver did not normally converge to a solution within the time limit even after trying multiple starting points, we report the best feasible solution found among all  (if any).

 Lastly, the experiments were carried out on a workstation with an Intel(R)
Core (TM) i7-4790 CPU processor at 3.60 GHz with 32.00 GB of RAM. All the methods
were coded in MATLAB(R2019b), using IBM ILOG CPLEX 12.9 for solving
linear programs.
\begin{table}[]
\tabcolsep=0.09cm
\renewcommand\arraystretch{0.4}
\small
\centering
\caption{Results related to the computation of $\underline{f}$ on class 1 instances (average gap and average running time)}
\label{tab:flb}
\begin{tabular}{@{}ccccccccccc@{}}
\toprule
\multicolumn{3}{c}{input} &  & \multicolumn{3}{c}{average gap} &  & \multicolumn{3}{c}{average time (sec)} \\ \cmidrule(r){1-3} \cmidrule(lr){5-7} \cmidrule(l){9-11} 
$m$ & $n$ & $\delta$ &  &\makecell{ LS\\$\underline{f}^{U_1}$} &\makecell{ FMINCON \\ $\underline{f}^{U_2}$} &\makecell{ super-set\\  $\underline{f}^{L}$ } &  & \makecell{ LS\\$\underline{f}^{U_1}$} & \makecell{ FMINCON \\ $\underline{f}^{U_2}$}  & \makecell{ super-set\\  $\underline{f}^{L}$ } \\ \midrule
10 & 15 & 0.1 &  & 0.0001 & 0.0035 & 0.0122 &  & 0.3214 & 0.4921 & 0.0016 \\
10 & 15 & 0.25 &  & 0.0001 & 0.0016 & 0.1530 &  & 0.3212 & 0.2684 & 0.0007 \\
10 & 15 & 0.5 &  & 0.0001 & 0.0186 & 0.1618 &  & 0.3411 & 0.3133 & 0.0007 \\
10 & 15 & 0.75 &  & 0.0001 & 0.0052 & 0.1411 &  & 0.3385 & 0.3034 & 0.0007 \\
10 & 15 & 1 &  & 0.0003 & 0.0190 & 0.9195 &  & 0.3437 & 0.3516 & 0.0007 \\ \midrule
30 & 45 & 0.1 &  & 0.0004 & 0.0057 & 0.1594 &  & 0.7657 & 1.6774 & 0.0020 \\
30 & 45 & 0.25 &  & 0.0029 & 0.0063 & 1.0221 &  & 0.8666 & 1.6139 & 0.0020 \\
30 & 45 & 0.5 &  & 0.0003 & 0.0094 & 0.5711 &  & 0.9188 & 1.5557 & 0.0020 \\
30 & 45 & 0.75 &  & 0.0004 & 0.0085 & 1.5103 &  & 0.9019 & 1.5185 & 0.0019 \\
30 & 45 & 1 &  & 0.0043 & 0.0175 & 1.4701 &  & 0.9346 & 1.6153 & 0.0020 \\ \midrule
50 & 75 & 0.1 &  & 0.0013 & 0.0045 & 0.3009 &  & 1.4564 & 3.2491 & 0.0042 \\
50 & 75 & 0.25 &  & 0.0018 & 0.0148 & 1.0008 &  & 1.5302 & 3.8956 & 0.0041 \\
50 & 75 & 0.5 &  & 0.0017 & 0.0084 & 0.9686 &  & 1.6549 & 4.3446 & 0.0040 \\
50 & 75 & 0.75 &  & 0.0066 & 0.0688 & 2.4971 &  & 1.6501 & 3.3910 & 0.0040 \\
50 & 75 & 1 &  & 0.0010 & 0.0640 & 4.3327 &  & 1.6596 & 3.6361 & 0.0040 \\ \midrule
80 & 120 & 0.1 &  & 0.0020 & 0.0381 & 1.0832 &  & 2.9608 & 12.1033 & 0.0106 \\
80 & 120 & 0.25 &  & 0.0142 & 0.0618 & 2.3559 &  & 3.1187 & 12.3475 & 0.0102 \\
80 & 120 & 0.5 &  & 0.0030 & 0.0227 & 1.4559 &  & 3.2310 & 12.5632 & 0.0102 \\
80 & 120 & 0.75 &  & 0.0018 & 0.0514 & 2.2257 &  & 3.2231 & 13.2636 & 0.0100 \\
80 & 120 & 1 &  & 0.0039 & 0.0510 & 2.3866 &  & 3.3486 & 11.3097 & 0.0099 \\ \midrule
100 & 150 & 0.1 &  & 0.0030 & 0.0646 & 1.5206 &  & 4.2595 & 26.5560 & 0.0163 \\
100 & 150 & 0.25 &  & 0.0034 & 0.0222 & 1.6112 &  & 4.5020 & 22.6072 & 0.0158 \\
100 & 150 & 0.5 &  & 0.0117 & 0.0504 & 2.6134 &  & 4.7064 & 23.4783 & 0.0154 \\
100 & 150 & 0.75 &  & 0.0153 & 0.0811 & 2.4725 &  & 4.7758 & 22.0584 & 0.0148 \\
100 & 150 & 1 &  & 0.0063 & 0.0563 & 3.6750 &  & 4.5380 & 23.5457 & 0.0151 \\ \bottomrule
\end{tabular}
\end{table}

\subsection{Analysis of the Results}

In this section, we only discuss the results related to $\underline{f}$. The analysis of the results for $\overline{f}$ led to similar conclusions, so we do not report them in the paper. Table \ref{tab:flb} shows the results related to the computation of $\underline{f}$ on class 1 instances, for which the optimal value can be computed by solving a linear program (see Proposition \ref{sovable1}). Each number in the table is an average of the results obtained on 30 instances. In the table, the first three columns show the input parameters, and the following six columns report the results of the solution approaches. We recall that the super-set based method returns a lower bound for $\underline{f}$ (columns $\underline{f}^L$), while the local search algorithm and   FMINCON  return an upper bound (columns $\underline{f}^{U_1}$ and $\underline{f}^{U_2}$, respectively). The gap of an approximate value $\hat{\underline{f}}$ from the optimal value is computed by $|\frac{\hat{\underline{f}}-\underline{f}}{\underline{f}}|$. Hence, lower values correspond to better performance of the approach. For each method, the table reports the average gap and the average running time (in seconds). 

The local search algorithm converges fast to a very tight upper bound (with a maximum average gap of  1.53\%  and a maximum average running time of  4.78 seconds) for all the problem sizes and all the values of the uncertainty parameter $\delta$. FMINCON returns a reasonable upper bound (with a maximum average gap of  8.11), but it takes significantly longer time than the local search algorithm to converge (with the average  running time ranges  between 0.27 of a second and 26.56 seconds). Although calculating $\underline{f}^L$ is fast, its gap from the the optimal value, with the exception of small size instances and low uncertainty, is significant. 
\begin{table}[]

\tabcolsep=0.2cm
\renewcommand\arraystretch{0.8}
\small
\centering
\caption{Results related to the computation of $\underline{f}$ on class 2 instances.}
\label{tab:nb}
\begin{tabular}{@{}cccccccccccc@{}}
\toprule
\multicolumn{3}{c}{input} &  & \multicolumn{2}{c}{$\underline{f}^{U_1}<\underline{f}^{U_2}$} &  & \multicolumn{2}{c}{$\underline{f}^{U_2}<\underline{f}^{U_1}$} &  & \multicolumn{2}{c}{average time (sec)} \\ \cmidrule(r){1-3} \cmidrule(lr){5-6} \cmidrule(lr){8-9} \cmidrule(l){11-12} 
$m$ & $n$ & $\delta$ &  & freq. & WAG &  & freq. & WAG &  & LS & FMINCON \\ \midrule
10 & 15 & 0.1 &  & 22 & 0.0024 &  & 8 & 0.0000 &  & 0.3130 & 0.4605 \\
10 & 15 & 0.25 &  & 21 & 0.0074 &  & 9 & 0.0000 &  & 0.3190 & 0.4511 \\
10 & 15 & 0.5 &  & 21 & 0.0045 &  & 9 & 0.0084 &  & 0.3301 & 0.5361 \\
10 & 15 & 0.75 &  & 19 & 0.0138 &  & 11 & 0.0002 &  & 0.3356 & 0.3225 \\
10 & 15 & 1 &  & 21 & 0.0526 &  & 9 & 0.0081 &  & 0.3536 & 0.3164 \\ \midrule
30 & 45 & 0.1 &  & 19 & 0.0107 &  & 11 & 0.0003 &  & 0.8241 & 1.3156 \\
30 & 45 & 0.25 &  & 19 & 0.0085 &  & 11 & 0.0010 &  & 0.8808 & 1.3838 \\
30 & 45 & 0.5 &  & 20 & 0.0397 &  & 10 & 0.0028 &  & 0.8978 & 1.4990 \\
30 & 45 & 0.75 &  & 17 & 0.0514 &  & 13 & 0.0084 &  & 0.9450 & 2.4027 \\
30 & 45 & 1 &  & 19 & 0.0390 &  & 11 & 0.0059 &  & 0.9791 & 1.8116 \\ \midrule
50 & 75 & 0.1 &  & 17 & 0.0031 &  & 13 & 0.0003 &  & 1.4768 & 4.8394 \\
50 & 75 & 0.25 &  & 15 & 0.0061 &  & 15 & 0.0089 &  & 1.5200 & 4.4078 \\
50 & 75 & 0.5 &  & 16 & 0.3897 &  & 14 & 0.0028 &  & 1.7290 & 4.0354 \\
50 & 75 & 0.75 &  & 16 & 0.0360 &  & 14 & 0.0067 &  & 1.7371 & 4.2858 \\
50 & 75 & 1 &  & 14 & 0.0585 &  & 16 & 0.1042 &  & 1.8498 & 4.9498 \\ \midrule
80 & 120 & 0.1 &  & 23 & 0.0141 &  & 7 & 0.0013 &  & 3.0263 & 13.7443 \\
80 & 120 & 0.25 &  & 17 & 0.0112 &  & 13 & 0.0364 &  & 3.4131 & 14.9026 \\
80 & 120 & 0.5 &  & 19 & 0.1273 &  & 11 & 0.0750 &  & 3.6254 & 13.2121 \\
80 & 120 & 0.75 &  & 16 & 0.0658 &  & 14 & 0.1080 &  & 3.5492 & 15.0886 \\
80 & 120 & 1 &  & 15 & 0.0841 &  & 15 & 0.0504 &  & 3.6229 & 19.0471 \\ \midrule
100 & 150 & 0.1 &  & 25 & 0.0143 &  & 5 & 0.0022 &  & 4.6070 & 30.4261 \\
100 & 150 & 0.25 &  & 23 & 0.1741 &  & 7 & 0.0109 &  & 4.9523 & 27.6289 \\
100 & 150 & 0.5 &  & 18 & 0.0413 &  & 12 & 0.0147 &  & 5.3517 & 31.3020 \\
100 & 150 & 0.75 &  & 18 & 0.0816 &  & 12 & 0.0337 &  & 5.6269 & 44.5924 \\
100 & 150 & 1 &  & 10 & 0.0904 &  & 20 & 0.0583 &  & 5.2683 & 38.2789 \\ \midrule
200 & 300 & 0.1 &  & 26 & 0.1033 &  & 4 & 0.0078 &  & 18.8387 & 226.2555 \\
200 & 300 & 0.25 &  & 15 & 0.2028 &  & 15 & 0.0120 &  & 20.0275 & 242.7648 \\
200 & 300 & 0.5 &  & 14 & 0.0415 &  & 16 & 0.1506 &  & 21.6897 & 298.3727 \\
200 & 300 & 0.75 &  & 9 & 0.0318 &  & 21 & 0.5077 &  & 20.1736 & 231.7975 \\
200 & 300 & 1 &  & 11 & 0.0254 &  & 19 & 0.2233 &  & 20.3656 & 203.9699 \\ \midrule
300 & 400 & 0.1 &  & 18 & 0.0655 &  & 12 & 0.0150 &  & 43.7986 & 1,400.9399 \\
300 & 400 & 0.25 &  & 18 & 0.1182 &  & 12 & 0.0424 &  & 46.3638 & 1,601.2650 \\
300 & 400 & 0.5 &  & 4 & 0.0539 &  & 26 & 0.2600 &  & 48.3374 & 1,536.8595 \\
300 & 400 & 0.75 &  & 6 & 0.0462 &  & 24 & 0.1886 &  & 39.8791 & 1,413.8520 \\
300 & 400 & 1 &  & 5 & 0.0434 &  & 25 & 1.0230 &  & 40.4988 & 1,351.4880 \\ \midrule
400 & 500 & 0.1 &  & 22 & 0.1410 &  & 8 & 0.0196 &  & 76.7331 & 1,740.2052 \\
400 & 500 & 0.25 &  & 17 & 3.3220 &  & 13 & 0.0188 &  & 76.2029 & 1,800 \\
400 & 500 & 0.5 &  & 9 & 0.0247 &  & 21 & 0.4430 &  & 76.8172 & 1,760.1362 \\
400 & 500 & 0.75 &  & 3 & 0.0063 &  & 27 & 0.8576 &  & 74.1061 & 1,784.4656 \\
400 & 500 & 1 &  & 2 & 0.0027 &  & 28 & 1.3027 &  & 70.8703 & 1,722.2246 \\ \midrule
500 & 600 & 0.1 &  & 22 & 0.0889 &  & 8 & 0.0274 &  & 124.1661 & 1,768.2018 \\
500 & 600 & 0.25 &  & 12 & 0.1292 &  & 18 & 0.0410 &  & 123.9820 & 1,800 \\
500 & 600 & 0.5 &  & 4 & 0.0224 &  & 26 & 0.3193 &  & 110.2257 & 1,770.3525 \\
500 & 600 & 0.75 &  & 4 & 0.0211 &  & 26 & 1.6025 &  & 112.0062 & 1,800 \\
500 & 600 & 1 &  & 2 & 0.0048 &  & 28 & 0.9238 &  & 108.6081 & 1,800 \\ \bottomrule
\end{tabular}
\end{table}

For class 2 instances, given the poor performance of the super-set based method, we only focus on the results obtained from the local search and the solver. As noted earlier, the unique B-stability property is not guarantied in class 2 instances, and as a result we are not able to solve ORP$_b$ to optimality using existing methods. We here compare  the local search and FMINCON  against each other. For instances where the local search outperforms the solver, i.e. $\underline{f}^{U_1}<\underline{f}^{U_2}$, we calculate the gap as $|\frac{\underline{f}^{U_1}-\underline{f}^{U_2}}{\underline{f}^{U_2}}|$, while for instances where the solver returns a better solution than the local search, namely $\underline{f}^{U_2}<\underline{f}^{U_1}$, the gap is determined by $|\frac{\underline{f}^{U_2}-\underline{f}^{U_1}}{\underline{f}^{U_1}}|$. Additionally, the following measure gives a weighted average gap (WAG) for each method. Specifically, it applies both the average gap and the number of times an algorithm outperforms the other, and reads
\begin{equation}\label{eq:wag}
\text{WAG}=\frac{(\text{number of instances on which an algorithm performs better})* (\text{average gap})}{\text{total number of instances}}.
\end{equation}  
Thus, the higher the WAG value, the better the performance. Table \ref{tab:nb} reports the results corresponding to the computation of $\underline{f}$ for class 2 instances. The  first three columns show  the input data. Columns 4 and 5 give the frequency of times the local search outperforms FMINCON and the weighted average gap, respectively. Similarly, the following two columns represent the same attributes for when  FMINCON outperforms the local search. The last two columns indicate the average running times for each method.

Our results in Table \ref{tab:nb} suggest that the local search outperformed FMINCON on 683 instances out of 1,350 instances. Moreover, the weighted average gap of  the local search is larger than that of FMINCON for 28 combinations of $m,n,\delta$ out of a total of 45 combinations. FMINCON tends to return a better weighted average gap for instances with larger sizes and uncertainty parameters, but it also takes a much longer time to converge to a solution, see, for example, the instances with $m=300$, $n=400$, and $\delta=1$. For this particular case,  FMINCON took 1,351.49 seconds on average to converge, while our local search converged, on average, in 40.5  seconds. From the table, we can  observe that the WAG measure ranges between 0.0024 and 3.3220  for the local search and  between 0.0000 and 1.6025 for FMINCON. We also see that the computation time of FMINCON grows faster compared to that of the local search, with a maximum average running time of 1,800 seconds against that of 124.17 seconds for the local search.

\section{Case Study: Healthcare Access Measurement} \label{case}
Here, we show an application of ORP$_b$ when an outcome function is used to measure spatial access to healthcare services. We first introduce a linear program which has been recently proposed in the literature to derive a matching  between patients and providers. We then use our approach to evaluate how uncertainty in input data influences spatial access to healthcare services, and discuss how the results of our approach can be used  for more reliable decision making.

\subsection{Optimization Model and Outcome Function}
Optimization models used to quantify potential spatial access to healthcare mimic the  interactions between two sets of actors in the system: the target population in need of service within each geographical area or community (e.g., census tract level), namely $e_i$ with $i \in T$, and the network of provider locations $j \in W$. Model 1  is a simplified version of  the mathematical formulation proposed in the literature \cite{gentili2015small,nobles2014spatial} to determine a matching between the population in need of healthcare services and providers providing them. The matching is determined to minimize the total distance traveled at the system level under a set of constraints: (i) \textit{coverage constraints} match as many people in need as possible; (ii) \textit{accessibility constraints} ensure the matching takes into account modes of transportation and Health Resources Services Administration recommendations on the maximum allowed distance for matching; (iii) \textit{capacity constraints} account for the maximum and minimum providers' caseload to stay in practice.

\begin{figure}[t]
\centering
\small
\singlespacing
\begin{tabular}{|lll|l|}
\multicolumn{3}{l}{{\normalsize{\bf Model 1.} Modeling access to primary care.}}\\
\hline
$\min  \sum_{i \in T, j \in W}g d_{ij}x_{ij} \label{eq:objective-function}$ &&&   	$\rightarrow$ Total distance is minimized.\\
subject to&&&\\
\hline
{\it Coverage constraints:}&&&\\
 $ \sum_{j\in W} x_{ij}\leq  e_i  $&$\forall i\in T, \label{eq:popcensus-constraints}$ &(C1) &$\rightarrow$  The assignment does not exceed population in need\\
&&&\hspace{4mm} in census tract $i$.\\
 $\sum_{i \in T, j\in W}x_{ij}\geq  \alpha E,  $& &(C2)&  $\rightarrow$ The assignment covers as much population as possible \\
&&&\hspace{4mm}within the national access policy.\\
\hline
{\it Accessibility constraints:}&&&\\
 $ \sum_{j\in W: d_{ij} \geq d_{max}}x_{ij} =0 $ & $ \forall i \in T, \label{eq:max-distance}$ &(C3)&$\rightarrow$ Patients are not assigned to too far providers. \\
 $ \sum_{j\in W: d_{ij} \geq d_{max}^{mob}}x_{ij} \leq m_ie_i $ & $ \forall i \in T, \label{eq:max-mob-distance}$ &  (C4)&$\rightarrow$ Patients that own a vehicle can  travel further than \\
&&&\hspace{4mm}patients without a vehicle.\\
\hline
 {\it Availability constraints:}&&&\\
 $ \sum_{i \in T}g x_{ij} \leq c^{max}_j  $&$\forall j \in W, \label{eq:max-caseload}$ &(C5)&$\rightarrow$ Providers' maximum caseload is not exceeded.\\
&&& \\
 $ \sum_{i \in T}g x_{ij} \geq c^{min}_j $&$ \forall j \in W ,\label{eq:min-caseload}$ &(C6)&$\rightarrow$ Providers are assigned a minimum caseload\\
&&&\hspace{4mm}to stay in practice.   \\
\hline
 {\it Non-negativity constraints:}&&&\\
 $x_{ij} \geq 0 $&$ \forall i \in T, \; j \in W. \label{eq:nonnegativity}$&& \\
\hline
\end{tabular}
\end{figure}

The decision variables $x_{ij}$ in the model determine the number of patients  in a census tract $i \in T$ assigned to a specific provider location $j \in W$. Parameters of the model include: 
\begin{itemize}
\item $g$:  number of yearly visits required by a patient,
\item $e_i$: population size in census tract $i$ in need of healthcare services,
\item $d_{ij}$: travel distance between the centroid of census tract $i$ and provider location $j$, 
\item $E$: total population in the system in need of healthcare services,
\item $\alpha$: percentage of the population which should be assigned to a provider, 
\item $d_{max}$: maximum allowed distance between a patient and the assigned provider according to the Health Resources Services Administration recommendations,  
\item $d_{max}^{mob}$: maximum distance we assume that people without a vehicle are willing to travel to reach the assigned provider, 
\item $m_i$: percentage of population in  census tract $i$ that owns a vehicle,
\item $c^{max}_j$ ($c^{min}_j$): maximum (minimum) provider's caseload in location $j$.
\end{itemize}

For our analysis, we consider an interval version of Model 1 obtained by allowing parameters $c^{max}_j$ to vary within a given interval. Specifically, we assume that the availability constraints (C5) in the model are of the form
 $$
 \sum_{i \in T}g x_{ij} \leq [\lambda c^{max}_j, \beta c^{max}_j ]\;\;\; \forall j \in W,
 $$
where $\lambda$ and $\beta$  are the maximum and the minimum perturbations from the nominal values $c^{max}_j$ for $ j \in W$, respectively. Note that the resulting intervals vary independently.  Such uncertainty in the capacity of a provider can be due to increasing and/or decreasing personnel, overtime or
days-off of providers, and inaccurate estimations of the capacity, among others.

Access measures are outcome functions defined as linear functions of an optimal assignment derived from an optimization model \cite{nobles2014spatial}. For this illustrative example, we consider the access measure $f_i(x)$ defined as the average distance traveled by patients in a given census tract $i$ to reach the assigned provider, which is formally defined as
\begin{equation*}
f_i(x)=d_{max}+\frac {1}{e_i} \sum_{j\in W}{(d_{ij}-d_{max})}x_{ij}\quad\forall i\in T.
\end{equation*}
The above measure gives the weighted average of the distance traveled by patients in each census tract. We assume that for those patients who are not assigned to a provider, $f_i(x)$ is equal to $d_{max}$. Thus, the access measure ranges from $0$ to $d_{max}$. 

The resulting estimates can be used by policy makers to identify where the communities with the greatest need for improvement are, so that they can be targeted with additional resources, including new providers or facilities, transportation services improvement, tele-health service development, etc. 
\subsection {Case Study}
We illustrate our analysis to quantify access to the primary care service for children in the State of Mississippi in the United States, for a total of 637 census tracts and 897 provider locations. Providers’ practice location addresses are obtained from the 2013 National Plan and Provider Enumeration System (NPPES). The patient population is aggregated at the census tract level. We used the 2010 SF2 100\% census data and the 2012 American Community Survey data to compute the number of children in each census tract along with information  on ownership of cars, to estimate the access to private transportation means. We set $d_{max}=25$ miles, $d_{max}^{mob}=10$ miles, $\alpha=0.85$, and $g=2$ (see \cite{gentili2018quantifying} for further details on the input parameters). The resulting model contains 63,573 variables and 3,706 constraints. For the interval version of the model, we set $\lambda=0.8$ and $\beta=1.2$. 

\subsection {Importance of Quantifying Sensitivity to Data Perturbations}
Failing to consider uncertainty in the input parameters may significantly affect  the decision making  on the choice of which census tracts to target for possible interventions. To elaborate further, we compared the results of Model 1 on two different realizations of interval data, referred to as realizations 1 \& 2.  Figure \ref{fig:ScenariosDifference}a shows the difference in the access measures obtained in the two optimization runs (corresponding to realizations 1 \& 2) . Darker regions represent higher differences, that is, census tracts where the estimate of the access measure is more unstable. The circled census tracts are those  for which the resulting access measure changes more than 5 miles between the two runs, implying that some census tracts may be considered having high or low level access depending on which realization of the data is considered.
Consider now Figure \ref{fig:ScenariosDifference}b where the difference in the access measures, obtained for two different additional realizations (referred to as realizations 3 \& 4) of the parameter $c^{max}_j$, is shown. The comparison between Figures \ref{fig:ScenariosDifference}a and \ref{fig:ScenariosDifference}b tells two different stories, showing completely different sets of census tracts for which the access measure seems more unstable.

In this sense, quantifying sensitivity of the access measure to data perturbations would be crucial for reliable decision making. Such an analysis would indeed reveal: (i) census tracts that are \textit{certainly} in need of a targeted intervention (e.g., those census tracts for which the access measure is high and not sensitive to data perturbations), and (ii) census tracts that are \textit{certainly not}  in need of any intervention (e.g., those census tracts for which the access measure is low and not sensitive to data perturbations). It would also help to  determine census tracts that may fall, due uncertainty in the data, in either one of the two categories, and for which, therefore, a deeper investigation might be needed.
By solving ORP$_b$ in this context, we can assess such a quantification. Additionally, we are able to answer questions relevant for policy making, including:

\begin{figure}
	\centering
	\begin{subfigure}{0.45\textwidth}
		\includegraphics[scale=0.45]{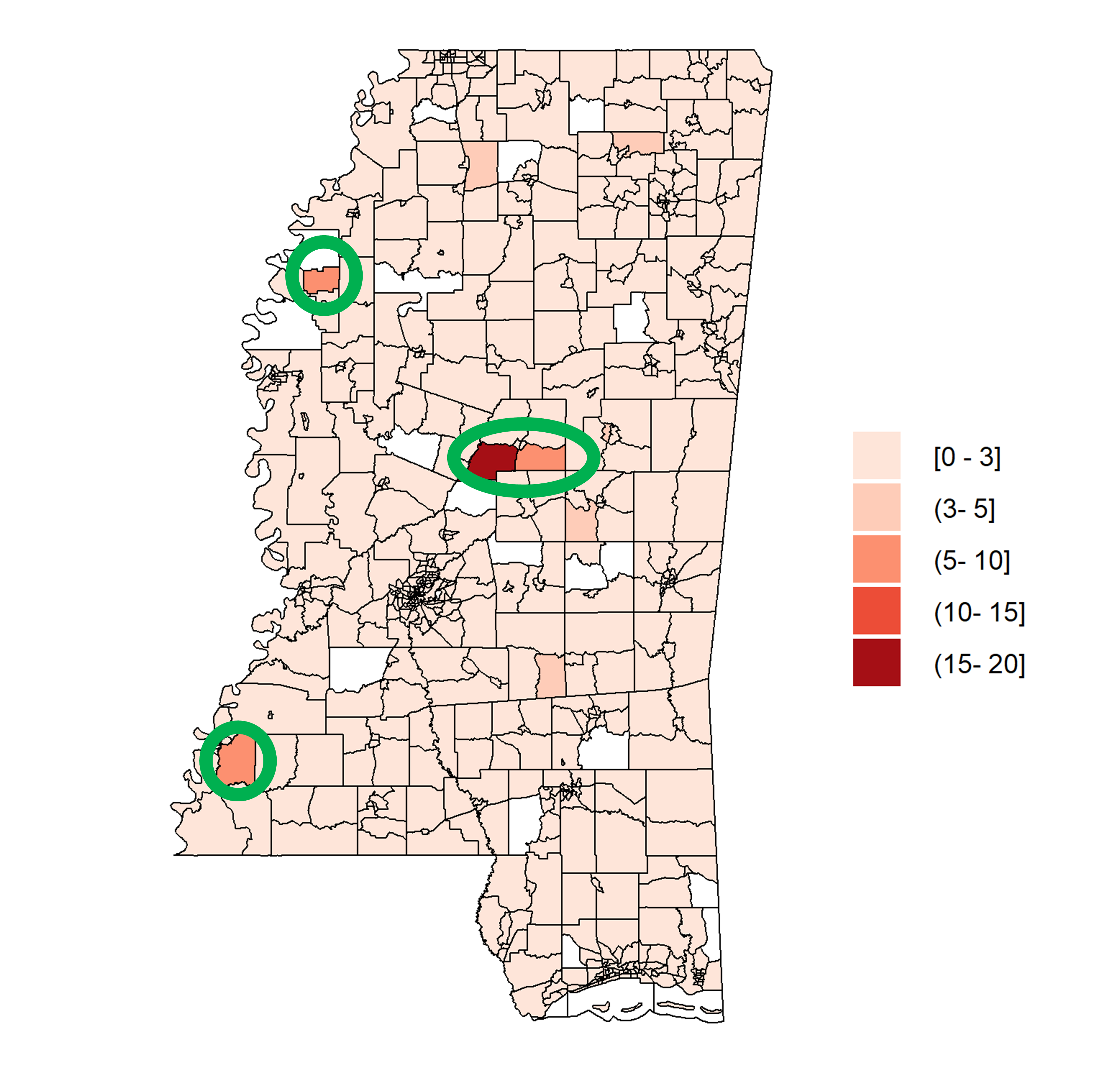}
		\caption{Realizations 1 \& 2}
	\end{subfigure}
	\begin{subfigure}{0.45\textwidth}
		\includegraphics[scale=0.45]{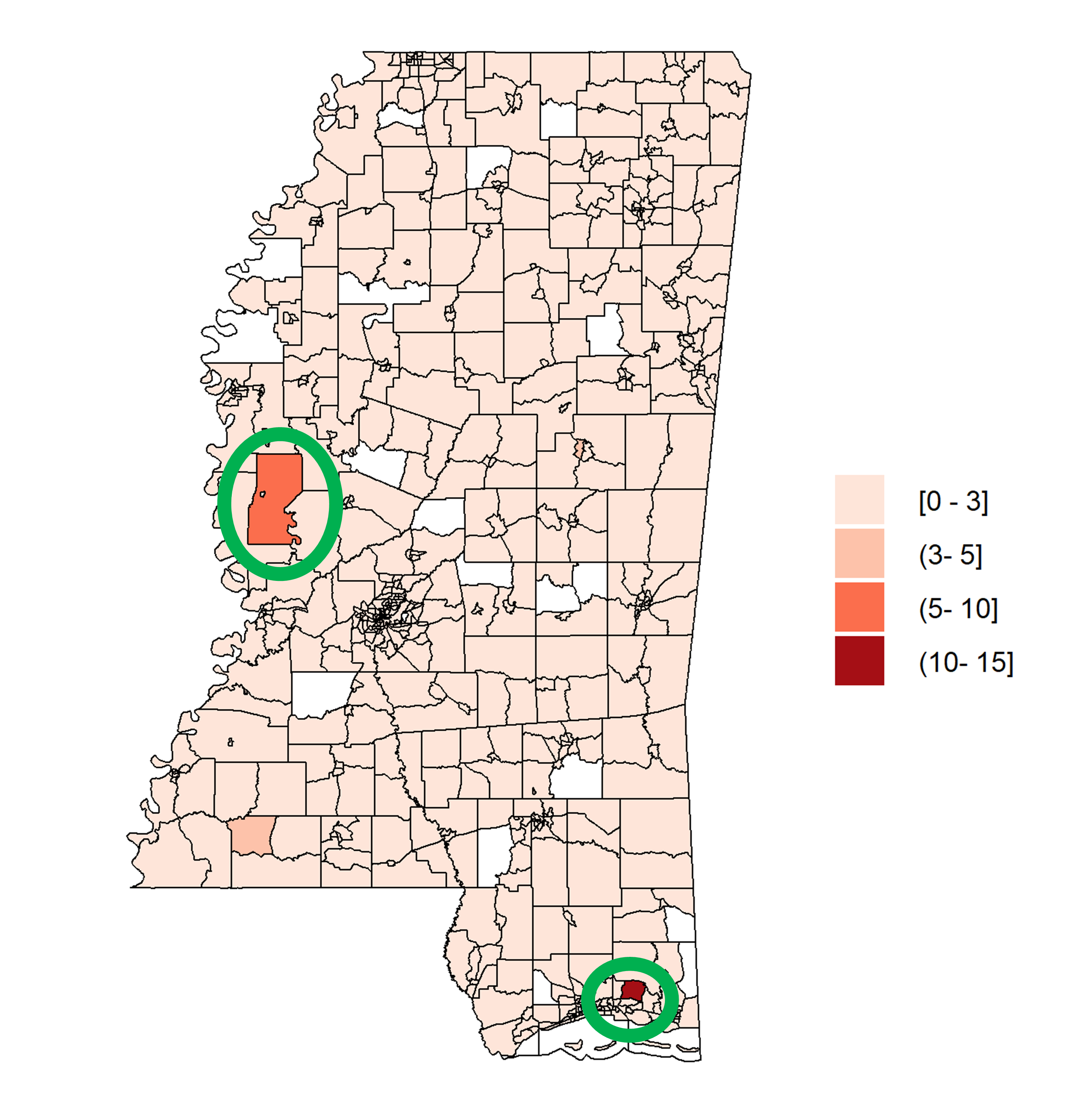}
		\caption{Realizations 3 \& 4}
	\end{subfigure}
	\caption{Difference in the access measures considering four random realizations of the input parameters.}
	\label{fig:ScenariosDifference}
\end{figure}

\begin{itemize}
\item Q1: Given the current primary care resources, what are the minimum and maximum access levels for each census tract? 
\item Q2: What are the census tracts with the highest (lowest) variability in the access measures? 
\item Q3: What is the percentage of the census tracts where the access level is higher (lower) than a given threshold for all the possible realizations of the data?

\end{itemize}
We applied our local search algorithm to solve ORP$_b$ in this context, and addressed the above questions.

\subsection{Implementation of Algorithms}
The outcome function (i.e., the access measure) is associated with each census tract. Hence, we applied our local search algorithm once for each outcome function (total of 637 functions). Zheng et al. \cite{zheng2017regularized} used the Monte Carlo approach to evaluate sensitivity of the access measure to uncertainty in the input data. Therefore, we compare the results of our approach with those returned by  the Monte Carlo approach.

For the local search algorithm, we defined  $Q=\{0.25,0.5,0.75,1\}$. Due to the large size of the problem and the structural dependencies among the decision variables \cite{zheng2017regularized}, defining an ordered set $V$ and randomly choosing constraints, whose right-hand sides are perturbed simultaneously, would not be very efficient. 
Thus, we defined a  set $V(i)$ for each given census tract $i$ as $V(i) = \{H_1(i), H_2(i), H_3(i)\}$ for all $i\in T$,  
where $H_l(i)$, $l=1,2,3$, are predefined sets of constraints associated with census tract $i$. Note that for this specific application, each constraint to be perturbed corresponds to a provider $j$ whose max capacity parameter ($c_j^{max}$) is perturbed from its nominal value. 
The first set of constraints to be explored corresponds to providers  who are not too far from the analyzed census tract, that is, $H_1(i) := \{j \in W: d_{ij} \leq 50 \}$.
The second set of constraints to be explored are those constraints corresponding to providers who do not correspond to constraints in $H_1(i)$ and who are not too far from census tracts which are neighbors of the  census tract under study. Specifically, we defined two census tracts to be neighbors if the distance between their centroids is less than 50 miles. Given a census tract $i$, let us denote the set of neighboring census tracts as the set $A(i) = \{ a \in T: d_{ia}\leq 50\}$. The second set of constraints is then defined as $H_2(i) :=\{j \in W: d_{aj}\leq 50, \forall a\in A(i)\} \backslash H_1(i)$. Finally, the last set $H_3(i)$ consists of the remaining providers, that is, $H_3(i) := W \backslash \{H_1(i) \cup H_2(i)\}$. 

We set the maximum number of shakes to 1 and the minimum acceptable improvement to 0.1. The number of iterations for the Monte Carlo approach was set equal to 100, which is the maximum number of linear programs  solved by the local search algorithm among all the runs. 
\subsection{Analysis of the Results}
Monte Carlo simulation is a simple approach to compute the maximum and the minimum values of the access measure for each census tract; however, in this context, it might lead to a severe underestimation of the overall quantification. To show this, we computed the range of the resulting access measure for each census tract using both the local search algorithm and the Monte Carlo approach. The range is computed as the difference between the maximum and the minimum of the access measure for each census tract. The difference in the results is shown in Figures \ref{fig:Bars}  and \ref{fig:MS_DiffNS_YlOrRd}.  
Specifically, Figure \ref{fig:Bars} shows the number of census tracts for which the range of the access measures is within 2 and 20 miles for the two approaches. Results obtained from the Monte Carlo approach show that the access range in 28 census tracts out of 635 census tracts varies between 4 and 20 miles, where in 12 of them the access range varies between 8 and 20 miles. However, the local search algorithm reveals that the access range in 89 census tracts varies between 4 and 20 miles, where in 47 of them the access range varies between 8 and 20 miles. It is noteworthy that Figure \ref{fig:Bars} does not represent census tracts with the access range of less than 2 miles.  
\begin{figure}[h]
\centering
	\includegraphics[scale=0.45]{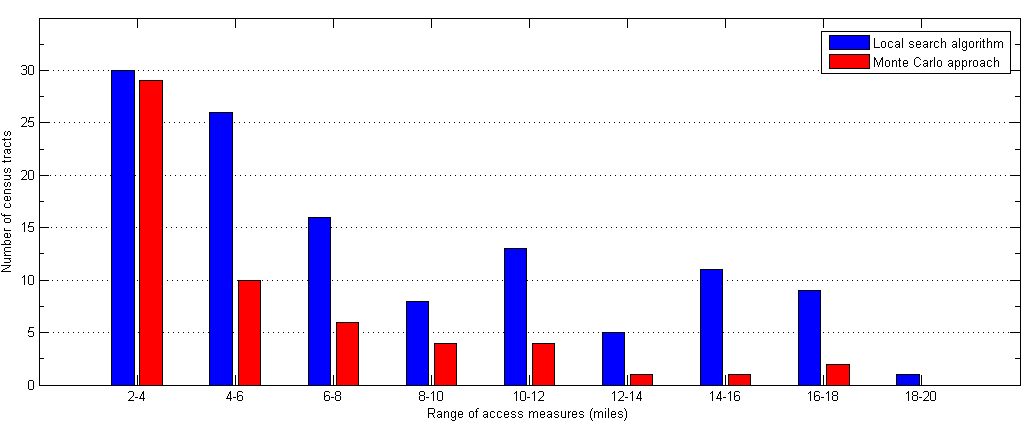}
	\caption{Distribution of the census tracts for which the access range varies between 2 and 20 miles for the two approaches (i.e., Monte Carlo approach and the local search algorithm).}
	\label{fig:Bars}
\end{figure}

\begin{figure}[h]
\centering
	\includegraphics[scale=0.20]{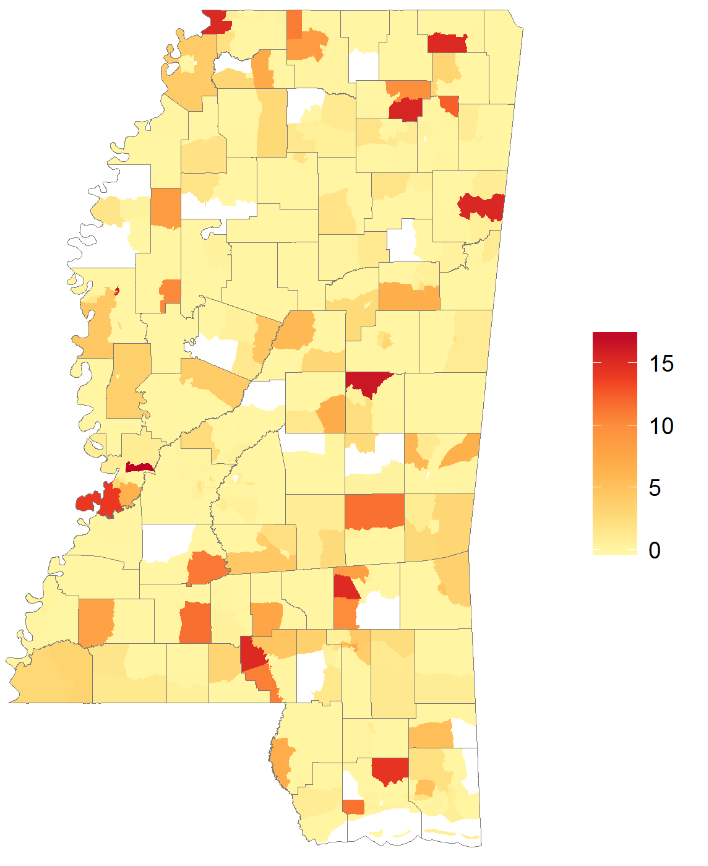}
	\caption{Difference between the access ranges estimated by the Monte Carlo approach and those estimated by the local search algorithm.}
	\label{fig:MS_DiffNS_YlOrRd}
\end{figure}

Figure \ref{fig:MS_DiffNS_YlOrRd} depicts the map of the difference in the ranges obtained comparing the two approaches. Darker census tracts are those for which the Monte Carlo approach severely underestimates sensitivity of the access measures, that is, those census tracts for which the difference between the range estimated by the Monte Carlo approach and the range estimated by the local search algorithm is greater than 15 miles. 
From  Figures \ref{fig:Bars} and \ref{fig:MS_DiffNS_YlOrRd}, it is evident that the Monte Carlo approach is not a right tool to quantify sensitivity of the access measure to uncertainty in the data. Its use to answer the questions Q1-Q3 would lead to a severe underestimation. Hence, in what follows, we only focus on the results obtained from our local search algorithm.

Figures \ref{fig:MinMaxAccess}a and \ref{fig:MinMaxAccess}b show the lower  and the upper limits of the access measure for each census tract (Q1), and Figure \ref{variation} shows the range of the access measure for each census tract (Q2). Darker areas in  Figure \ref{variation} are those census tracts where the range of the access measure is greater than 10 miles, which corresponds to 39 census tracts out of 635 (i.e., 6\% of the total). Table \ref{analysis} and Figure \ref{fig:maxmin} can be used to address question Q3. The table shows the distribution of the minimum and maximum of access within the state among all the census tracts.
Figure \ref{fig:maxmin}a divides the census tracts in two groups according to the value of their minimum access level: dark (light) tracts have a minimum access which is greater (less than or equal to)  10 miles. 
Figure \ref{fig:maxmin}b  divides the census tracts in two groups according to the value of their maximum access level: dark (light) tracts have a maximum access which is greater than (less than or equal to)  5 miles. 
 According to Table \ref{analysis},  13\% of the census tracts have a minimum level of access which is greater than 10 miles. In other words, the population in these census tracts always travel on average at least 10 miles to reach the assigned provider. These census tracts are the dark regions in Figure \ref{fig:maxmin}a. On the other hand, 64\% of the census tracts (the column maximum level of access in Table  \ref{analysis}) are such that the corresponding population never travel more than 5 miles to reach the assigned provider. These census tracts are the light regions in Figure \ref{fig:maxmin}b. 
These findings are important for decision makers to prioritize interventions. Indeed, for example, dark census tracts in Figure \ref{fig:maxmin}a depict those census tracts which are surely in need for targeted actions to improve their access to healthcare services because  they were identified by accounting for all the possible realizations of the uncertain data, while the light census tracts in Figure \ref{fig:maxmin}b have a good access to healthcare services among all the possible realizations of the uncertain data; hence, they are unlikely to be the object of targeted interventions.

\begin{figure}[H]

	\begin{subfigure}{0.45\textwidth}
		\includegraphics[scale=0.40]{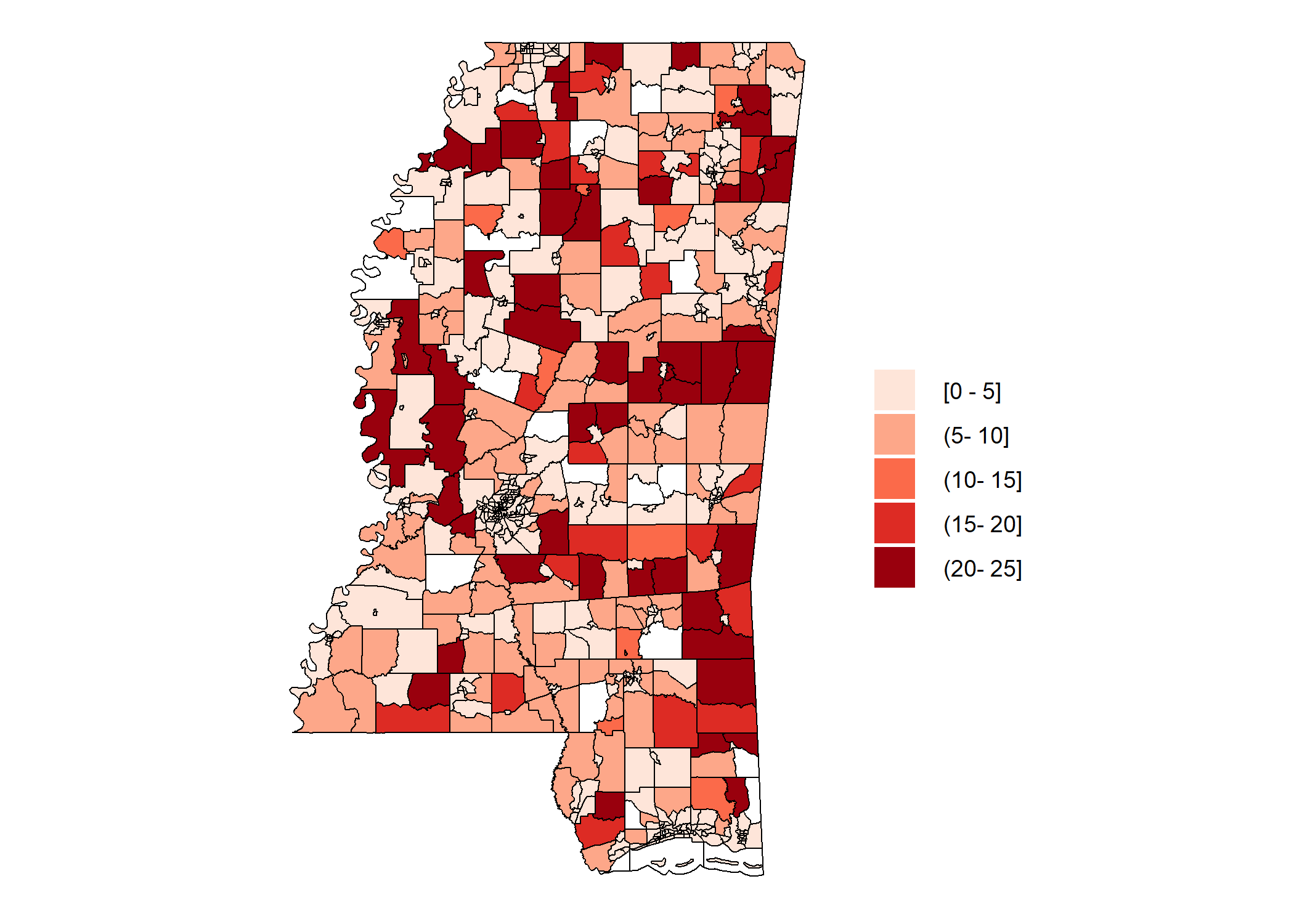}
		\caption{Minimum access level}	\end{subfigure}
	\begin{subfigure}{0.45\textwidth}
		\includegraphics[scale=0.40]{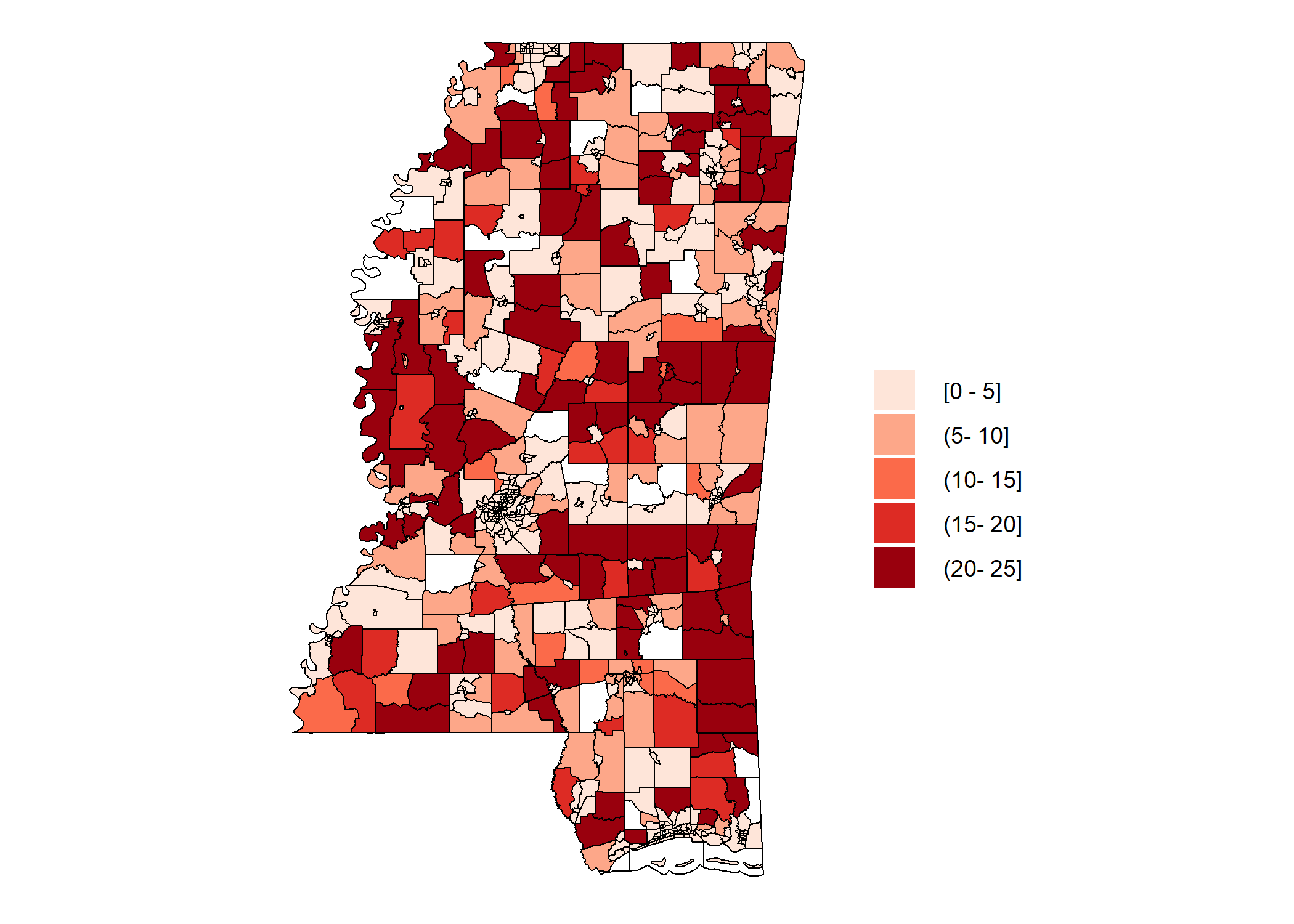}
		\caption{Maximum access level}
	\end{subfigure}
\caption{Minimum and maximum  of the access measures.}
	\label{fig:MinMaxAccess}
\end{figure}
\begin{figure}[H]
\centering
	\includegraphics[scale=0.40]{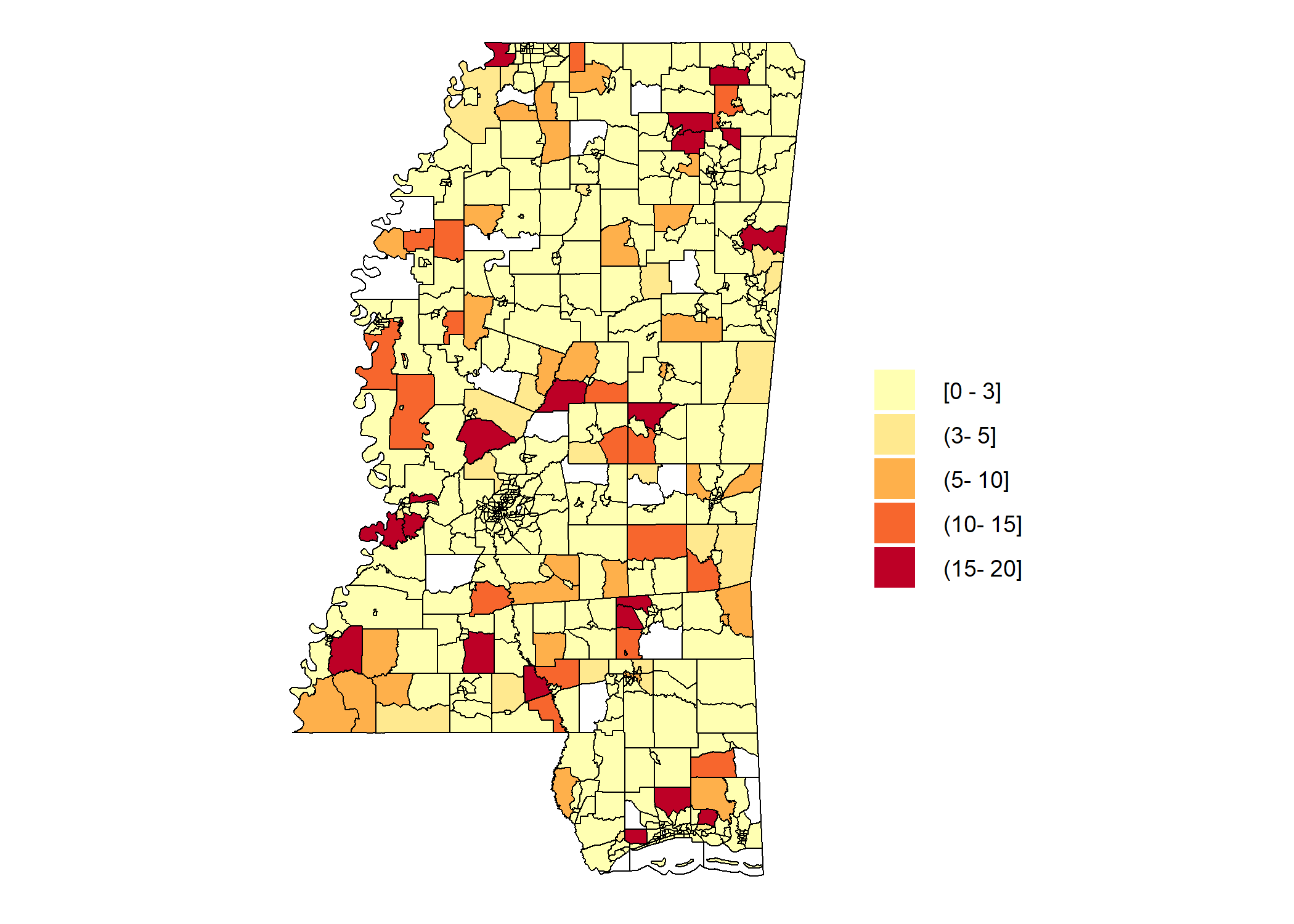}
	\caption{Range of the access measure for each census tract.}
	\label{variation}
\end{figure}

\begin{table}[H]
\small
\centering
\caption{Distribution of census tracts corresponding to the minimum and maximum access levels (for different access ranges).}
\begin{tabular}{@{}ccc@{}}

\toprule
access (mile) & minimum access level & maximum access level \\ \midrule
0-5                  & 69\%                    & 64\%                    \\
5-10                 & 18\%                    & 14\%                    \\
10-15                & 2\%                     & 2\%                     \\
15-20                & 3\%                     & 4\%                     \\
20-25                & 8\%                     & 15\%                    \\ \bottomrule
\label{analysis}
\end{tabular}
\end{table}
\begin{figure}[H]
	\centering
	\begin{subfigure}{0.45\textwidth}
		\includegraphics[scale=0.40]{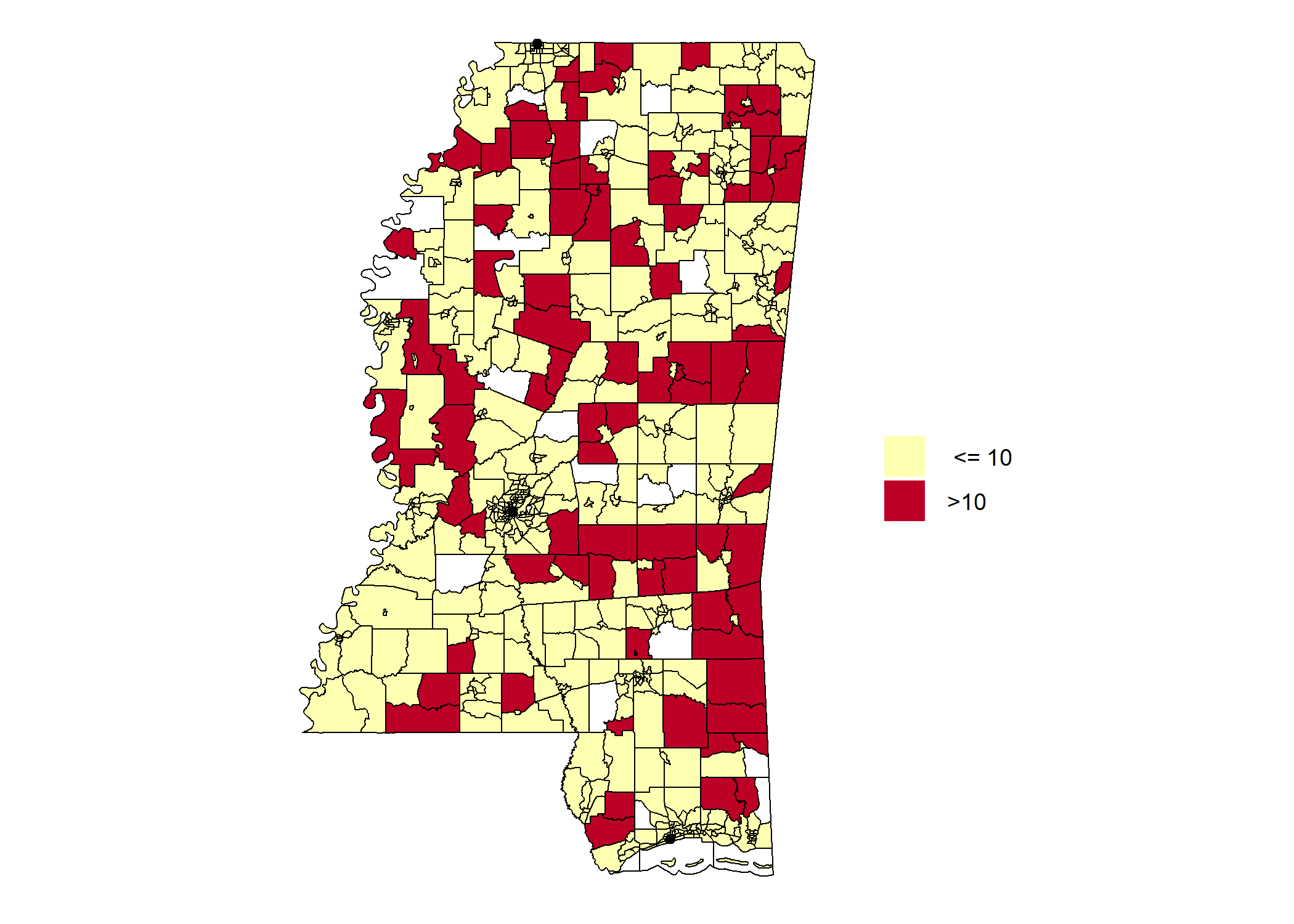}
		\caption{Minimum access level}
	\end{subfigure}
	\begin{subfigure}{0.45\textwidth}
		\includegraphics[scale=0.40]{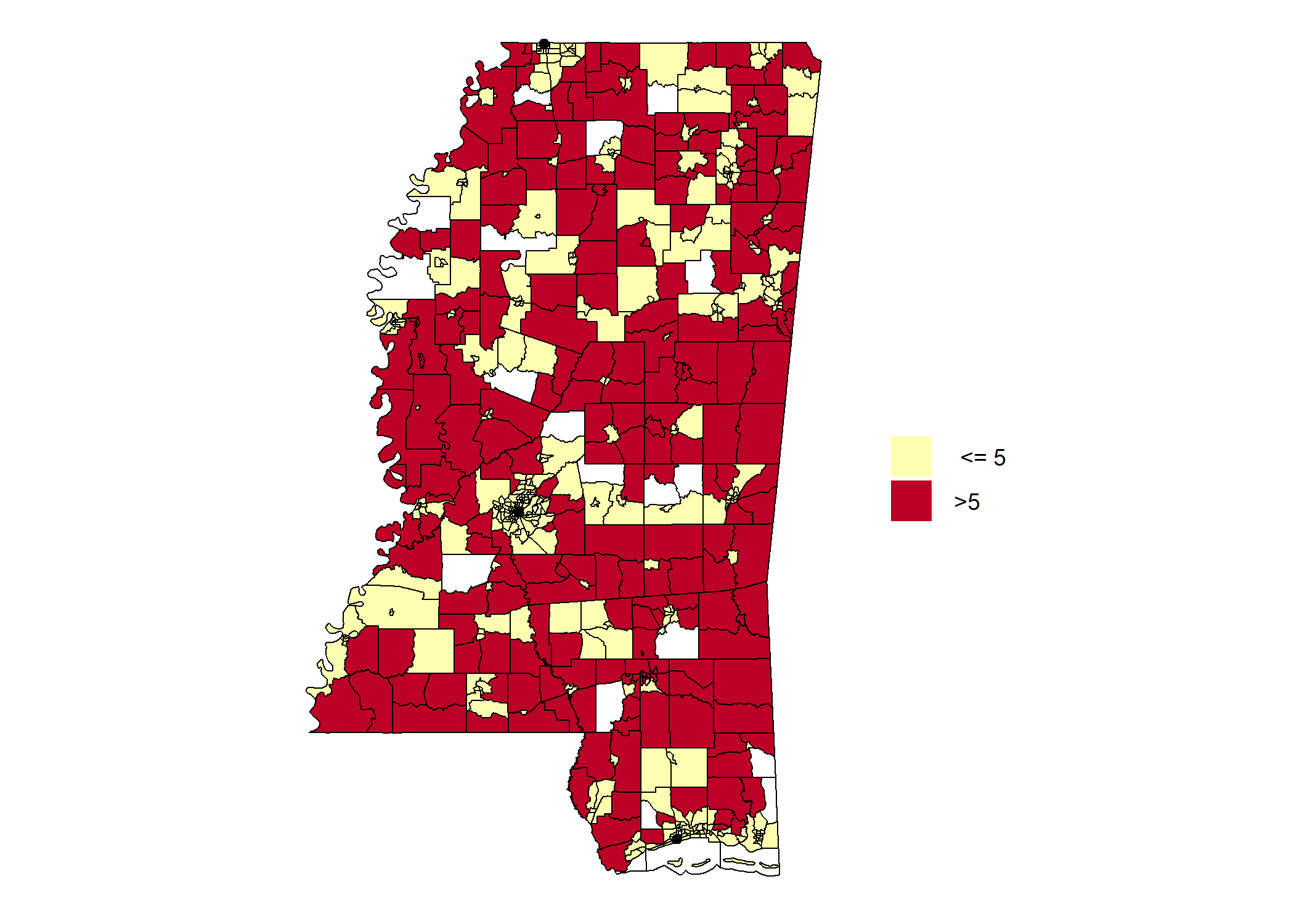}
		\caption{Maximum access level}
	\end{subfigure}
	\caption{Classification of the census tracts according to their minimum and maximum access levels.}
	\label{fig:maxmin}
\end{figure}
\section{Conclusions} \label{conclusion}
We formulated and studied  the {\it  outcome range problem} in the field of interval linear programming. Our problem aims at quantifying unintended consequences of an optimal decision in an uncertain environment. The problem is particularly relevant for government agencies, public health decision makers, policy makers, city managers and other stakeholders who make decisions that have differential impacts on different
communities and sub-populations, and we showed this on a real case study related to healthcare access measurement. In this paper, we gave a very general definition of the outcome range problem, and addressed a specific version of it for which we assessed the computational complexity and studied some theoretical properties. We then offered two approximation methods. Our proposed local search algorithm seems promising in computing a cheap but  tight approximation of the problem. In contrast, the proposed super-set based method does not return a tight approximation; thus, there is room for improvement.
We tested the methods on two sets of randomly generated instances, and on a real case instance. We plan to further explore theoretical properties and solution methods for a more general version of the problem.


\bibliographystyle{apa}
\bibliography{COR_V7}

\begin{thebibliography}{}

\bibitem[\protect\astroncite{Allahdadi and Nehi}{2013}]{allahdadi2013optimal}
Allahdadi, M. and Nehi, H.~M. (2013).
\newblock The optimal solution set of the interval linear programming problems.
\newblock {\em Optimization Letters}, 7(8):1893--1911.
\newblock \url{https://doi.org/10.1007/s11590-012-0530-4}.

\bibitem[\protect\astroncite{Beeck}{1978}]{beeck1978linear}
Beeck, H. (1978).
\newblock Linear programming with inexact data.
\newblock Technical Report TUM-ISU-7830, Technical University of Munich,
  Munich.

\bibitem[\protect\astroncite{Bertsimas and Sim}{2004}]{bertsimas2004price}
Bertsimas, D. and Sim, M. (2004).
\newblock The price of robustness.
\newblock {\em Operations Research}, 52(1):35--53.
\newblock \url{https://doi.org/10.1287/opre.1030.0065}.

\bibitem[\protect\astroncite{Cerulli et~al.}{2017}]{itp}
Cerulli, R., D’Ambrosio, C., and Gentili, M. (2017).
\newblock Best and worst values of the optimal cost of the interval
  transportation problem.
\newblock In {\em International Conference on Optimization and Decision
  Science}, pages 367--374. Springer.
\newblock \url{https://doi.org/10.1007/978-3-319-67308-0_37}.

\bibitem[\protect\astroncite{Chinneck and Ramadan}{2000}]{chinneck2000linear}
Chinneck, J. and Ramadan, K. (2000).
\newblock Linear programming with interval coefficients.
\newblock {\em Journal of the Operational Research Society}, 51(2):209--220.
\newblock \url{https://doi.org/10.1057/palgrave.jors.2600891}.

\bibitem[\protect\astroncite{D’Ambrosio et~al.}{2020}]{d2020optimal}
D’Ambrosio, C., Gentili, M., and Cerulli, R. (2020).
\newblock The optimal value range problem for the interval (immune)
  transportation problem.
\newblock {\em Omega}, 95:102059.
\newblock \url{https://doi.org/10.1016/j.omega.2019.04.002}.

\bibitem[\protect\astroncite{Gabrel et~al.}{2010}]{gabrel2010linear}
Gabrel, V., Murat, C., and Remli, N. (2010).
\newblock Linear programming with interval right hand sides.
\newblock {\em International Transactions in Operational Research},
  17(3):397--408.
\newblock \url{https://doi.org/10.1111/j.1475-3995.2009.00737.x}.

\bibitem[\protect\astroncite{Garajov{\'a}}{2016}]{thesis}
Garajov{\'a}, E. (2016).
\newblock The optimal solution set of interval linear programming problems.
\newblock Thesis, Charles Univesity, Faculty of Mathematics and Physics,
  Prague,Czech Republic.
  \url{https://is.cuni.cz/webapps/zzp/detail/168259/?lang=en}.

\bibitem[\protect\astroncite{Garajov{\'a} and
  Hlad{\'\i}k}{2019}]{garajova2019optimal}
Garajov{\'a}, E. and Hlad{\'\i}k, M. (2019).
\newblock On the optimal solution set in interval linear programming.
\newblock {\em Computational Optimization and Applications}, 72(1):269--292.
\newblock \url{https://doi.org/10.1007/s11590-012-0530-4}.

\bibitem[\protect\astroncite{Garajov{\'a} et~al.}{2019}]{garajova2018interval}
Garajov{\'a}, E., Hlad{\'\i}k, M., and Rada, M. (2019).
\newblock Interval linear programming under transformations: optimal solutions
  and optimal value range.
\newblock {\em Central European Journal of Operations Research},
  27(3):601--614.
\newblock \url{https://doi.org/10.1007/s10100-018-0580-5}.

\bibitem[\protect\astroncite{Gentili et~al.}{2016}]{gentili2016projecting}
Gentili, M., Harati, P., and Serban, N. (2016).
\newblock Projecting the impact of the affordable care act provisions on
  accessibility and availability of primary care providers for the adult
  population in georgia.
\newblock {\em American Journal of Public Health}, 106(8):1470--1476.
\newblock \url{https://doi.org/10.2105/AJPH.2016.303222}.

\bibitem[\protect\astroncite{Gentili et~al.}{2018}]{gentili2018quantifying}
Gentili, M., Harati, P., Serban, N., O'connor, J., and Swann, J. (2018).
\newblock Quantifying disparities in accessibility and availability of
  pediatric primary care across multiple states with implications for targeted
  interventions.
\newblock {\em Health Services Research}, 53(3):1458--1477.
\newblock \url{https://doi.org/10.1111/1475-6773.12722}.

\bibitem[\protect\astroncite{Gentili et~al.}{2015}]{gentili2015small}
Gentili, M., Isett, K., Serban, N., and Swann, J. (2015).
\newblock Small-area estimation of spatial access to care and its implications
  for policy.
\newblock {\em Journal of Urban Health}, 92(5):864--909.
\newblock \url{https://doi.org/10.1007/s11524-015-9972-1}.

\bibitem[\protect\astroncite{Hlad{\'\i}k}{2009}]{hladik2009optimal}
Hlad{\'\i}k, M. (2009).
\newblock Optimal value range in interval linear programming.
\newblock {\em Fuzzy Optimization and Decision Making}, 8(3):283--294.
\newblock \url{https://doi.org/10.1007/s10700-009-9060-7}.

\bibitem[\protect\astroncite{Hlad{\'\i}k}{2012a}]{survey}
Hlad{\'\i}k, M. (2012a).
\newblock Interval linear programming: A survey.
\newblock In {\em In Chapter 2. In: Mann ZA (ed) Linear programming—new
  frontiers in theory and applications}, pages 85--120. Nova Science
  Publishers, New York.

\bibitem[\protect\astroncite{Hlad{\'\i}k}{2012b}]{contractor}
Hlad{\'\i}k, M. (2012b).
\newblock An interval linear programming contractor.
\newblock {\em In: J. Ramik and D. Stavarek (eds.), Proceedings 30th Int. Conf.
  Mathematical Methods in Economics 2012, Karvina, Czech Republic}, pages
  284--289.
\newblock Silesian University in Opava.

\bibitem[\protect\astroncite{Hlad{\'\i}k}{2013}]{hladi2013weak}
Hlad{\'\i}k, M. (2013).
\newblock Weak and strong solvability of interval linear systems of equations
  and inequalities.
\newblock {\em Linear Algebra and its Applications}, 438(11):4156--4165.
\newblock \url{https://doi.org/10.1016/j.laa.2013.02.012}.

\bibitem[\protect\astroncite{Hlad{\'\i}k}{2014a}]{hladik2014determine}
Hlad{\'\i}k, M. (2014a).
\newblock How to determine basis stability in interval linear programming.
\newblock {\em Optimization Letters}, 8(1):375--389.
\newblock \url{https://doi.org/10.1007/s11590-012-0589-y}.

\bibitem[\protect\astroncite{Hlad{\'\i}k}{2014b}]{hladik2014approximation}
Hlad{\'\i}k, M. (2014b).
\newblock On approximation of the best case optimal value in interval linear
  programming.
\newblock {\em Optimization Letters}, 8(7):1985--1997.
\newblock \url{https://doi.org/10.1007/s11590-013-0715-5}.

\bibitem[\protect\astroncite{Hlad{\'\i}k}{2018}]{hladik2018worst}
Hlad{\'\i}k, M. (2018).
\newblock The worst case finite optimal value in interval linear programming.
\newblock {\em Croatian Operational Research Review}, 9(2):245--254.
\newblock \url{https://doi.org/10.17535/crorr.2018.0019}.

\bibitem[\protect\astroncite{Hlad\'{\i}k}{2020}]{10.1145/3396474.3396479}
Hlad\'{\i}k, M. (2020).
\newblock Two approaches to inner estimations of the optimal solution set in
  interval linear programming.
\newblock In {\em Proceedings of the 2020 4th International Conference on
  Intelligent Systems, Metaheuristics \& Swarm Intelligence}, ISMSI ’20, page
  99–104, New York, NY, USA. Association for Computing Machinery.
\newblock \url {https://doi.org/10.1145/3396474.3396479}.

\bibitem[\protect\astroncite{Huang et~al.}{1995}]{huang1995grey}
Huang, G.~H., Baetz, B.~W., and Patry, G.~G. (1995).
\newblock Grey integer programming: an application to waste management planning
  under uncertainty.
\newblock {\em European Journal of Operational Research}, 83(3):594--620.
\newblock \url{https://doi.org/10.1016/0377-2217(94)00093-R}.

\bibitem[\protect\astroncite{Jansson and Rump}{1991}]{jansson1991rigorous}
Jansson, C. and Rump, S.~M. (1991).
\newblock Rigorous solution of linear programming problems with uncertain data.
\newblock {\em Zeitschrift f{\"u}r Operations Research}, 35(2):87--111.
\newblock \url{https://doi.org/10.1007/BF02331571}.

\bibitem[\protect\astroncite{Juman and Hoque}{2014}]{juman2014heuristic}
Juman, Z. and Hoque, M. (2014).
\newblock A heuristic solution technique to attain the minimal total cost
  bounds of transporting a homogeneous product with varying demands and
  supplies.
\newblock {\em European Journal of Operational Research}, 239(1):146--156.

\bibitem[\protect\astroncite{Kumar et~al.}{2016}]{kumar2016interval}
Kumar, P., Panda, G., and Gupta, U. (2016).
\newblock An interval linear programming approach for portfolio selection
  model.
\newblock {\em International Journal of Operational Research},
  27(1-2):149--164.
\newblock \url{https://doi.org/10.1504/IJOR.2016.078463}.

\bibitem[\protect\astroncite{Lai et~al.}{2002}]{lai2002class}
Lai, K.~K., Wang, S., Xu, J., Zhu, S., and Fang, Y. (2002).
\newblock A class of linear interval programming problems and its application
  to portfolio selection.
\newblock {\em IEEE Transactions on Fuzzy Systems}, 10(6):698--704.
\newblock \url{https://doi.org/10.1109/TFUZZ.2002.805902}.

\bibitem[\protect\astroncite{Li}{2016}]{li2016interval}
Li, D. (2016).
\newblock Interval-valued matrix games.
\newblock In {\em Linear Programming Models and Methods of Matrix Games with
  Payoffs of Triangular Fuzzy Numbers}, pages 3--63. Springer.

\bibitem[\protect\astroncite{Liu and Kao}{2009}]{liu2009matrix}
Liu, S.-T. and Kao, C. (2009).
\newblock Matrix games with interval data.
\newblock {\em Computers and Industrial Engineering}, 56(4):1697--1700.
\newblock \url{https://doi.org/10.1016/j.cie.2008.06.002}.

\bibitem[\protect\astroncite{McCormick}{1976}]{McCormick1976}
McCormick, G.~P. (1976).
\newblock Computability of global solutions to factorable nonconvex programs:
  Part i --- convex underestimating problems.
\newblock {\em Mathematical Programming}, 10(1):147--175.
\newblock \url{https://doi.org/10.1007/BF01580665}.

\bibitem[\protect\astroncite{Mohammadi and Gentili}{2019}]{soco}
Mohammadi, M. and Gentili, M. (2019).
\newblock Bounds on the worst optimal value in interval linear programming.
\newblock {\em Soft Computing}, 23(21):11055--11061.
\newblock \url{https://doi.org/10.1007/s00500-018-3658-z}.

\bibitem[\protect\astroncite{Mraz}{1998}]{mraz1998calculating}
Mraz, F. (1998).
\newblock Calculating the exact bounds of optimal values in lp with interval
  coefficients.
\newblock {\em Annals of Operations Research}, 81(1):51--62.
\newblock \url{https://doi.org/10.1023/A:1018985914065}.

\bibitem[\protect\astroncite{Nobles et~al.}{2014}]{nobles2014spatial}
Nobles, M., Serban, N., and Swann, J. (2014).
\newblock Spatial accessibility of pediatric primary healthcare: measurement
  and inference.
\newblock {\em The Annals of Applied Statistics}, 8(4):1922--1946.
\newblock \url{https://doi.org/10.1214/14-AOAS728 }.

\bibitem[\protect\astroncite{Novotn{\'a} et~al.}{2020}]{novotna2020duality}
Novotn{\'a}, J., Hlad{\'\i}k, M., and Masa{\v{r}}{\'\i}k, T. (2020).
\newblock Duality gap in interval linear programming.
\newblock {\em Journal of Optimization Theory and Applications},
  184(2):565--580.
\newblock \url{https://doi.org/10.1007/s10957-019-01610-y}.

\bibitem[\protect\astroncite{Oettli and Prager}{1964}]{oettli1964compatibility}
Oettli, W. and Prager, W. (1964).
\newblock Compatibility of approximate solution of linear equations with given
  error bounds for coefficients and right-hand sides.
\newblock {\em Numerische Mathematik}, 6(1):405--409.
\newblock \url{https://doi.org/10.1007/BF01386090}.

\bibitem[\protect\astroncite{Park et~al.}{2018}]{park2016environmental}
Park, Y.~S., Lim, S.~H., Egilmez, G., and Szmerekovsky, J. (2018).
\newblock Environmental efficiency assessment of u.s. transport sector: A
  slack-based data envelopment analysis approach.
\newblock {\em Transportation Research Part D: Transport and Environment},
  61(1):152 -- 164.
\newblock \url{https://doi.org/10.1016/j.trd.2016.09.009}.

\bibitem[\protect\astroncite{Rohn}{2006a}]{rohn1}
Rohn, J. (2006a).
\newblock Interval linear programming.
\newblock In {\em Linear Optimization Problems with Inexact Data}, pages
  35--77. Springer.
\newblock \url{https://doi.org/10.1007/0-387-32698-7_3}.

\bibitem[\protect\astroncite{Rohn}{2006b}]{rohn2}
Rohn, J. (2006b).
\newblock Solvability of systems of interval linear equations and inequalities.
\newblock In {\em Linear Optimization Problems with Inexact Data}, pages
  35--77. Springer.
\newblock \url{https://doi.org/10.1007/0-387-32698-7_2}.

\bibitem[\protect\astroncite{Rohn}{2012}]{rohn2005handbook}
Rohn, J. (2012).
\newblock A handbook of results on interval linear problems.
\newblock Technical Report 1163. Institute of Computer Science. Academy of
  Sciences of the Czech Republic. Prague.
  \url{http://www.nsc.ru/interval/Library/Surveys/ILinProblems.pdf}.

\bibitem[\protect\astroncite{Siarry}{2016}]{siarry2016metaheuristics}
Siarry, P. (2016).
\newblock {\em Metaheuristics}.
\newblock Springer.
\newblock \url{ https://doi.org/10.1007/978-3-319-45403-0}.

\bibitem[\protect\astroncite{Soyster}{1973}]{soyster1973convex}
Soyster, A.~L. (1973).
\newblock Convex programming with set-inclusive constraints and applications to
  inexact linear programming.
\newblock {\em Operations Research}, 21(5):1154--1157.
\newblock \url{https://doi.org/10.1287/opre.21.5.1154}.

\bibitem[\protect\astroncite{Vajda}{1961}]{vajda1961mathematical}
Vajda, S. (1961).
\newblock {\em Mathematical programming}.
\newblock Addison-Wesley, Reading Mass., USA.

\bibitem[\protect\astroncite{Wang and Huang}{2014}]{wang2014violation}
Wang, X. and Huang, G. (2014).
\newblock Violation analysis on two-step method for interval linear
  programming.
\newblock {\em Information Sciences}, 281:85--96.
\newblock \url{https://doi.org/10.1016/j.ins.2014.05.019}.

\bibitem[\protect\astroncite{Winston et~al.}{2003}]{winston2003introduction}
Winston, W.~L., Venkataramanan, M., and Goldberg, J.~B. (2003).
\newblock {\em Introduction to mathematical programming}, volume~1.
\newblock Thomson/Brooks/Cole Duxbury; Pacific Grove, CA.

\bibitem[\protect\astroncite{Zheng et~al.}{2018}]{zheng2017regularized}
Zheng, Y., Lee, I., and Serban, N. (2018).
\newblock Regularized optimization with spatial coupling for robust decision
  making.
\newblock {\em European Journal of Operational Research}, 270(3):898--906.
\newblock \url{https://doi.org/10.1016/j.ejor.2017.10.037}.

\bibitem[\protect\astroncite{Zhou et~al.}{2009}]{zhou2009enhanced}
Zhou, F., Huang, G.~H., Chen, G.-X., and Guo, H.-C. (2009).
\newblock Enhanced-interval linear programming.
\newblock {\em European Journal of Operational Research}, 199(2):323--333.
\newblock \url{https://doi.org/10.1016/j.ejor.2008.12.019}.

\bibitem[\protect\astroncite{Zhou et~al.}{2014}]{zhou2014measuring}
Zhou, G., Chung, W., and Zhang, Y. (2014).
\newblock Measuring energy efficiency performance of china’s transport
  sector: A data envelopment analysis approach.
\newblock {\em Expert Systems with Applications}, 41(2):709--722.
\newblock \url{https://doi.org/10.1016/j.eswa.2013.07.095}.

\end{thebibliography}

\end{document}